\documentclass[10pt]{amsart}

\usepackage[colorlinks]{hyperref}
\usepackage{color,graphicx,shortvrb}
\usepackage[latin 1]{inputenc}

\usepackage[active]{srcltx} 

\usepackage{enumerate}
\usepackage{amssymb}

\newtheorem{theorem}{Theorem}[section]

\newtheorem{corollary}[theorem]{Corollary}

\newtheorem{lemma}[theorem]{Lemma}

\newtheorem{problem}[theorem]{Problem}
\newtheorem{open problem}[op]{Open Problem}
\newtheorem{proposition}[theorem]{Proposition}
\newtheorem{remark}[theorem]{Remark}

\def\11{\textbf{$1$}}

\DeclareGraphicsExtensions{.jpg,.pdf,.png,.eps}

\begin{document}

\title[A survey on Tingley's problem for operator algebras]{A survey on Tingley's problem for operator algebras}

\author[A.M. Peralta]{Antonio M. Peralta}

\address{Departamento de An{\'a}lisis Matem{\'a}tico, Facultad de
Ciencias, Universidad de Granada, 18071 Granada, Spain.}
\email{aperalta@ugr.es}


\subjclass[2010]{Primary 47B49, Secondary 46A22, 46B20, 46B04, 46A16, 46E40.}

\keywords{Tingley's problem; extension of isometries; von Neumann algebra; $p$-Schatten von Neumann, trace class operators}

\date{}

\begin{abstract} We survey the most recent results on extension of isometries between special subsets of the unit spheres of C$^*$-algebras, von Neumann algebras, trace class operators, preduals of von Neumann algebras, and $p$-Schatten-von Neumann spaces, with special interest on Tingley's problem.
\end{abstract}

\maketitle
\thispagestyle{empty}

\section{Introduction}

The problem of extending a surjective isometry between two subsets of the unit spheres of two operator algebras was treated in several talks during the conference on preserver problems held in Szeged in June 2017. The conference ``Preservers Everywhere'' gathered a substantial group of world experts on preservers problems. It became clear that the problems regarding the extension of this type of surjective isometries constitute an intensively studied line in recent times. Let us try to unify all these problems in the following statement.

\begin{problem}\label{problem general} Let $X$ and $Y$ be two Banach spaces whose unit spheres are denoted by $S(X)$ and $S(Y)$, respectively.
Let $\mathcal{S}_1$ and $\mathcal{S}_2$ be two subsets of $S(X)$ and $S(Y)$, respectively. Suppose $\Delta : \mathcal{S}_1 \to \mathcal{S}_2$ is a surjective isometry. Does $\Delta$ extend to a real linear isometry from $X$ onto $Y$?
\end{problem}

Henceforth, we shall write $\mathbb{T}$ for the unit sphere of $\mathbb{C}$. The complex conjugation on $\mathbb{T}$ cannot be extended to a complex linear isometry on $\mathbb{C}$. So, in the case of complex Banach spaces, a complex linear extension is simply hopeless for all cases. Similar constrains will appear in subsequent results.\smallskip

These problems, whose origins are in geometry, are nowadays a central topic for those researchers working on preservers. If in Problem \ref{problem general} we consider $\mathcal{S}_1= S(X)$ and $\mathcal{S}_2= S(Y)$ we meet the so-called \emph{Tingley's problem}. This problem was named after the contribution of D. Tingley, who established that for any two finite dimensional Banach spaces $X$ and $Y$, every surjective isometry $\Delta:S(X)\to S(Y)$ preserves antipodal points, that is, $\Delta(-x) = -\Delta(x),$ for every $x$ in $S(X)$ (see \cite[THEOREM in page 377]{Ting1987}). Tingley's problem remains open even in the case of two dimensional Banach spaces.\smallskip

Let us observe that, given a surjective isometry $\Delta:S(X)\to S(Y)$, where $X$ and $Y$ are Banach spaces, we can always consider the natural (positively) homogeneous extension $F_{\Delta} : X\to Y$ given by $F_{\Delta}(0)=0,$ and $F_{\Delta}(x) = \|x\| \Delta\left(\frac{x}{\|x\|}\right)$ for $x\neq 0$. Clearly, $F_{\Delta}$ is a bijection, however it is a hard question to decide whether $F_{\Delta}$ is an isometry. Actually, the Mazur-Ulam theorem implies that $F_{\Delta}$ is real linear as soon as it is an isometry.\smallskip

We have already found our first connection with the Mazur-Ulam theorem. Tingley's problem and Problem \ref{problem general} can be considered as generalization of this pioneering result in Functional Analysis. P. Mankiewicz established in 1972 an intermediate result which provides an useful tool for our purposes.

\begin{theorem}\label{t Mankiewicz}\cite[Theorem 5 and Remark 7]{Mank1972} Every bijective isometry between convex sets in normed linear spaces with nonempty interiors admits a unique extension to a bijective affine isometry between the corresponding spaces.
\end{theorem}

During the thirty years elapsed after Tingley's paper, a lot of hard efforts from many authors, especially many Chinese mathematicians, and the elite Chinese group leaded by G.G. Ding, have been conducted in the seeking of a solution to Tingley's problem in concrete spaces. The huge contribution due to mathematicians like R.S. Wang, G.G. Ding, D. Tan, L. Cheng, Y. Dong, X.N. Fang, J.H. Wang, and R. Liu, among others, have been overview in full detail in the excellent surveys published by G.G. Ding \cite{Ding2009} and X. Yang and X. Zhao \cite{YangZhao2014}.\smallskip

A reborn interest on the problems concerning extension of isometries between subsets of the unit spheres of two operator algebras has been materialized in a fruitful series of recent papers dealing with Tingley's problem and related questions for certain operator algebras, which have been published during the short interval determined by the last three years. The abundance of new results for operator algebras motivates and justifies the writing of this survey with the aim of completing and updating the surveys \cite{Ding2009,YangZhao2014}, and providing a recent state of the art of these problems. The real ``avalanche'' of recent achievements provides enough material to write a new and detailed survey on this topic.\smallskip

We strive for conciseness and for restrict the results to the setting of operator algebras, despite that some of the results have been already extended to the strictly wider setting of JB$^*$-triples (compare \cite{FerPe17,FerPe17c}). So, few or none proofs are explicitly included. The main tools and results are reviewed with a full bibliographic information. We shall also insert some new arguments to establish some additional statements.\smallskip

In section \ref{sec: geometric background} we gather some of the key tools applied in many of the proofs given to solve Tingley's problem. Most of the studies make use of a result, which was originally established by L. Cheng, Y. Dong in \cite{ChenDong2011}, and proves that a surjective isometry $\Delta: S(X) \to S(Y)$ between the unit spheres of two Banach spaces, maps maximal proper faces of the closed unit ball of $X$ to maximal proper proper faces of the closed unit ball of $Y$ (see Theorem \ref{t ChengDong for general faces}). The section also contains a recent generalization of this result due to F.J. Fen{\'a}ndez-Polo, J. Garc{\'e}s, I. Villanueva and the author of this note, which assures the following: Let $\Delta: S(X) \to S(Y)$ be a surjective isometry between the unit spheres of two Banach spaces, and suppose that these spaces satisfy the following two properties: \begin{enumerate}[$(h.1)$]\item Every norm closed face of $\mathcal{B}_X$ {\rm(}respectively, of $\mathcal{B}_Y${\rm)} is norm-semi-exposed;
\item Every weak$^*$ closed proper face of $\mathcal{B}_{X^*}$ {\rm(}respectively, of $\mathcal{B}_{Y^*}${\rm)} is weak$^*$-semi-exposed.
\end{enumerate} Then the following statements hold:
\begin{enumerate}[$(a)$]\item Let $\mathcal{F}$ be a convex set in $S(X)$. Then $\mathcal{F}$ is a norm closed face of $\mathcal{B}_{X}$ if and only if $\Delta(\mathcal{F})$ is a norm closed face of $\mathcal{B}_{Y}$;
\item Let $e\in S(X)$. Then $e\in \partial_e (\mathcal{B}_X)$ if and only if $\Delta(e)\in \partial_e (\mathcal{B}_Y)$
\end{enumerate} (see Corollary \ref{c for spaces with property of semi-exposition for faces}).\smallskip

It should be remarked that hypotheses $(h.1)$ and $(h.2)$ above hold whenever $X$ and $Y$ are C$^*$-algebras, hermitian parts of C$^*$-algebras, von Neumann algebra preduals, preduals of the hermitian part of a von Neumann algebra, JB$^*$-triples, and JBW$^*$-triple preduals  (see \cite{FerGarPeVill17} and the comments after Corollary \ref{c for spaces with property of semi-exposition for faces}).\smallskip

In section \ref{sec: geometric background} we shall also survey the main results on the facial structure of the closed unit ball of a C$^*$-algebra due to C.A. Akemann and G.K. Pedersen \cite{AkPed92} and C.M. Edwards and G.T. R\"{u}ttimann \cite{EdRutt88}.\smallskip

Section \ref{sec: Tingleys Cstar} is completely devoted to present the most recent achievements on Tingley's problem in the setting of C$^*$-algebras. In all sections we shall insert an introductory paragraph with the equivalent results in the commutative setting. We begin from the results by R. Tanaka, which assure that every surjective isometry from the unit sphere of a finite dimensional C$^*$-algebra $A$ onto the unit sphere of another C$^*$-algebra $B$ extends to a unique surjective real linear isometry from $A$ onto $B$, and the same conclusion holds when $A$ and $B$ are finite von Neumann algebras (see Theorem \ref{t Tanaka finite vN}). In Theorem \ref{thm Tingley compact Cstaralgebras} we revisit the solution to Tingley's problem for surjective isometries between the unit spheres of two compact C$^*$-algebras found by R. Tanaka and the author of this survey in \cite{PeTan16}. This solution also covers the case of a surjective isometry between the unit spheres of two $K(H)$ spaces. In this note $K(H)$ and $B(H)$ will denote the spaces of compact and bounded linear operators on a complex Hilbert space $H$, respectively.\smallskip

Accordingly to the chronological order, the next step in the study of Tingley's problem on C$^*$-algebras is a result by F.J. Fern{\'a}ndez-Polo and the author of this note, which shows that for any two complex Hilbert spaces $H_1$ and $H_2$, every surjective isometry $\Delta: S(B(H_1)) \to S(B(H_2))$ admits a unique extension to a surjective complex linear or conjugate linear surjective isometry $T$ from $B(H_1)$ onto $B(H_2)$ satisfying $\Delta(x) = T(x)$, for every $x\in S(B(K))$ (see Theorem \ref{t Tingley BH spaces}). The most conclusive result on Tingley's problem has been also obtained by the same authors in a result showing that every surjective isometry $\Delta : S(M)\to S(N)$ between the unit spheres of two von Neumann algebras admits a unique extension to a surjective real linear isometry $T: M\to N$. Furthermore, under these hypotheses, there exist a central projection $p$ in $N$ and a Jordan $^*$-isomorphism $J : M\to N$ such that defining $T : M\to N$ by $T(x) = \Delta(1) \left( p J (x) + (1-p) J (x)^*\right)$ {\rm(}$x\in M${\rm)}, then $T$ is a surjective real linear isometry and $T|_{S(M)} = \Delta$ (see Theorem \ref{t Tingley von Neumann}).\smallskip

Section \ref{sec: Tingleys predual} is devoted to survey the results on Tingley's problem for surjective isometries between the unit spheres of von Neumann algebra preduals. In \cite{FerGarPeVill17}, F.J. Fern{\'a}ndez-Polo, J. Garc{\'e}s, I. Villanueva and the author of this survey gave a complete solution to Tingley's problem for surjective isometries on the unit sphere of the space $C_1(H)$ of trace class operators on an arbitrary complex Hilbert space $H$ (see Theorem \ref{t Tingley for trace class infinite dimension}).\smallskip

It is well known that the space $C_1(H)$ identifies with the dual of the space $K(H)$ and with the predual of $B(H)$. It seems a natural question whether the previous positive solution to Tingley's problem in the setting of trace class operators remains true for preduals of general von Neumann algebras.\smallskip

When the writing of this survey was being completed (precisely, on December 27th, 2017), an alert message came to this author from arxiv. This alert was about a very recent preprint by M. Mori (see \cite{Mori2017}), which has been an impressive discovering, and made this autor change the original project to insert some nice achievements, one of them is a positive solution to Tingley's problem for preduals of general von Neumann algebras (see Theorem \ref{t Tingleys problem for preduals of vN}).\smallskip

Henceforth, the hermitian part of a C$^*$-algebra $A$ will be denoted by $A_{sa}$. As we commented before, when in Problem \ref{problem general} the subsets $\mathcal{S}_1$ and $\mathcal{S}_2$ are the unit spheres of two Banach spaces, we find the so-called Tingley's problem. Another interesting variant of Problem \ref{problem general} is obtained when $X$ and $Y$ are von Neumann algebras or C$^*$-algebras and $\mathcal{S}_1$ and $\mathcal{S}_2$ are the unit spheres of their respective hermitian parts. In Section \ref{sec: Tingleys hermitian}, we shall study the problem of extending a  surjective isometry $\Delta: S(M_{sa})\to S(N_{sa})$, where $M$ and $N$ are von Neumann algebras. In this section we shall show that the same tools given by F.J. Fern{\'a}ndez-Polo and the author of this survey in \cite{FerPe17d} can be, almost literarily, applied to find a surjective complex linear isometry $T: M \to N$ satisfying $T(a^*) = T(a)^*$ for all $a$ in $M$ and $T(x) = \Delta (x)$ for all $x$ in $S(M_{sa})$ (see Theorem \ref{t Tingley hermitian von Neumann}).\smallskip

It should be remarked here that, after completing the writing of this chapter, the preprint by M. Mori \cite{Mori2017} became available in arxiv. Section 5 in \cite{Mori2017} is devoted to the study of Tingley's problem for surjective isometries between the unit spheres of the hermitian parts of two von Neumann algebras, and our Theorem \ref{t Tingley hermitian von Neumann} is also established by M. Mori with a alternative proof.\smallskip

The sixth and final section of this paper is devoted to review the main result on a topic which had its own protagonism in the meeting held in Szeged. We are talking about the problem of extending a surjective isometry between the sets of positive norm-one operators of two type I von Neumann factors $B(H_1)$ and $B(H_2)$. During the talk presented by G. Nagy in this conference, he presented a recent achievement which shows that for a finite dimensional complex Hilbert space $H$, every isometry $\Delta : S(B(H)^+)\to S(B(H)^+)$ admits a (unique) extension to a surjective complex linear isometry $T : B(H) \to B(H)$ satisfying $T(x) = \Delta(x)$ for all $x\in S(B(H)^+)$ (see Theorem \ref{t Nagy B(H) positive finite dim}), where for a C$^*$-algebra $A$, the symbol $A^+$ will denote the cone of positive elements in $A$, and $S(A^+)$ will stand for the sphere of positive norm-one operators. It was conjectured by Nagy that the same conclusion holds for every complex Hilbert space $H$.\smallskip

We culminate this section, and the results in this note, by surveying a recent work where we provide a proof to Nagy's conjecture. The main result is treated in Theorem \ref{t positive Tigley for B(H)}, where it is shown that every surjective isometry $\Delta : S(B(H_1)^+)\to S(B(H_2)^+)$, where $H_1$ and $H_2$ are complex Hilbert spaces, admits an extension to a surjective complex linear isometry {\rm(}actually, a $^*$-isomorphism or a $^*$-anti-automorphism{\rm)} $T: B(H_1)\to B(H_2)$.\smallskip

We shall revisit one of the main tools employed to establish the above result. This tool is a geometric characterization of projections in atomic von Neumann algebras. Let us recall some notation first. Suppose that $E$ and $P$ are non-empty subsets of a Banach space $X$. Following the notation employed in the recent paper \cite{Per2017}, the \emph{unit sphere around $E$ in $P$} is the set $$Sph(E;P) :=\left\{ x\in P : \|x-b\|=1 \hbox{ for all } b\in E \right\}.$$ To simplify the notation, given a C$^*$-algebra $A$, and a subset $E\subset A$, we shall write $Sph^+ (E)$ or $Sph_A^+ (E)$ for the set $Sph(E;S(A^+))$. The geometric characterization of projections reads as follows: let $M$ be an atomic von Neumann algebra, and let $a$ be a positive norm-one element in $M$. Then the following statements are equivalent: \begin{enumerate}
[$(a)$] \item $a$ is a projection; \item $Sph^+_{M} \left( Sph^+_{M}(a) \right) =\{a\}$.
\end{enumerate} (see Theorem \ref{t characterization of projection in terms of the sphere around a positive element}).
This characterization also holds when $M$ is replaced by $K(H_3)$, where $H_3$ is a separable complex Hilbert space (Theorem \ref{t characterization of projection in terms of the sphere around a positive element k(H)}). Moreover, if $a$ is a positive norm-one element in an arbitrary C$^*$-algebra $A$ satisfying $Sph^+_{A} \left( Sph^+_{A}(a) \right) =\{a\}$, then $a$ is a projection (see \cite[Proposition 2.2]{Per2017}).\smallskip

This geometric characterization has been also applied to prove that if $H_3$ and $H_4$ are separable complex Hilbert spaces, then every surjective isometry $$\Delta : S(K(H_3)^+)\to S(K(H_4)^+)$$ admits a unique extension to a surjective complex linear isometry $T$ from  $K(H_3)$ onto $K(H_4)$ (see Theorem \ref{t Nagy for K(ell2)}).

\section{Geometric background}\label{sec: geometric background}

In this section we survey the basic geometric tools which are frequently applied in most of the studies extending isometries. The results gathered in this section are established in the general setting of Banach spaces.\smallskip

A non-empty convex subset $F$ of a convex set $C$ is said to be a \emph{face} of $C$ if $\alpha x+ (1-\alpha ) y\in F$ with $x,y\in C$ and $0<\alpha<1$ implies $x,y\in F$. An element $x$ in the unit sphere of a Banach space $X$ is an \emph{extreme point} of $\mathcal{B}_{X}$ precisely when the set $\{x\}$ is a face of $\mathcal{B}_{X}$. Accordingly to the standard notation, from now on, the extreme points of a convex set $C$ will be denoted by $\partial_e(C)$. The Krein-Milman theorem is a fantastic tool to assure the existence and abundance of extreme points in any non-empty compact convex subset of a locally convex, Hausdorff, topological vector space.\smallskip

Up to now, most of the studies on Tingley's problem are based on a good and appropriate knowledge of the geometric properties of the involved spaces. This is because the most general geometric conclusion which can be derived from the existence of a surjective isometry between the unit spheres of two Banach spaces is the following result, which was originally established by L. Cheng and Y. Dong \cite{ChenDong2011}, and later rediscovered by R. Tanaka \cite{Tan2016,Tan2014}. From now on, given a normed space $X$, the symbol $\mathcal{B}_X$ will stand for the closed unit ball of $X$.

\begin{theorem}\label{t faces ChengDong11}{\rm(}\cite[Lemma 5.1]{ChenDong2011}, \cite[Lemma 3.3]{Tan2016}, \cite[Lemma 3.5]{Tan2014}{\rm)} Let $\Delta: S(X) \to S(Y)$ be a surjective isometry between the unit spheres of two Banach spaces, and let $\mathcal{M}$ be a convex subset of $S(X)$. Then $\mathcal{M}$ is a maximal proper face of $\mathcal{B}_X$ if and only if $\Delta(\mathcal{M})$ is a maximal proper {\rm(}closed{\rm)} face of $\mathcal{B}_Y$.
\end{theorem}

As we commented at the introduction, Tingley's problem remains open even in the case of two dimensional Banach spaces, the reason, probably, being the lacking of a concrete description of the maximal convex subsets of the unit sphere of a general Banach space.\smallskip

All strategies based on Theorem \ref{t faces ChengDong11} above require a concrete description of the maximal proper norm-closed faces of $\mathcal{B}_X$ in terms of the algebraic or geometric properties of $X$. This is the point where the results of C.A. Akemann and G.K. Pedersen \cite{AkPed92}, C.M. Edwards and G.T. R\"{u}ttimann \cite{EdRutt88}, C.M. Edwards, F.J. Fern{\'a}ndez-Polo, C. Hoskin and A.M. Peralta \cite{EdFerHosPe2010}, and F.J. Fern{\'a}ndez-Polo and A.M. Peralta \cite{FerPe10}, describing the facial structure of the closed unit ball of C$^*$-algebras, von Neumann algebra preduals, JB$^*$-triples and their dual spaces, and JBW$^*$-triples and their preduals, become an useful tool.\smallskip

We recall now the ``facear'' and ``pre-facear'' operations introduced in \cite{EdRutt88}. For each $F\subseteq \mathcal{B}_X$  and $G\subseteq \mathcal{B}_{X^*}$, we define $$ F^{\prime} = \{a \in \mathcal{B}_{X^*}:a(x) = 1\,\, \forall x \in F\},\quad
G_{\prime} = \{x \in \mathcal{B}_X :a(x) = 1\,\, \forall a \in G\}.$$ Then, $F^{\prime}$ is a weak$^*$ closed face of $\mathcal{B}_{X^*}$ and $G_{\prime}$
is a norm closed face of $\mathcal{B}_X$. The subset $F$ is said to
be a \emph{norm-semi-exposed face} of $\mathcal{B}_X$ if $F=(F^{\prime})_{\prime}$, while the subset $G$ is called a \emph{weak$^*$-semi-exposed face} of $\mathcal{B}_{X^*}$ if $G =(G_{\prime})^{\prime}$. The mappings $F \mapsto F^{\prime}$
and $G \mapsto G_{\prime}$ are anti-order isomorphisms between the
complete lattices $\mathcal{S}_n(\mathcal{B}_X)$ of norm-semi-exposed faces
of $\mathcal{B}_X,$ and $\mathcal{S}_{w^*}(\mathcal{B}_{X^*})$ of weak$^*$-semi-exposed
faces of $\mathcal{B}_{X^*}$ and are inverses of each other.\smallskip

If in Theorem \ref{t faces ChengDong11} we assume a richer geometric structure on the spaces $X$ and $Y$, then the conclusion of this result was improved in a recent paper by F.J. Fern{\'a}ndez-Polo, J. Garc{\'e}s, I. Villanueva and the author of this note in \cite{FerGarPeVill17}.\smallskip

\begin{theorem}\label{t ChengDong for general faces}{\rm\cite[Proposition 2.4]{FerGarPeVill17}} Let $\Delta: S(X) \to S(Y)$ be a surjective isometry between the unit spheres of two Banach spaces, and let $C$ be a convex subset of $S(X)$. Suppose that for every extreme point $\phi_0$ in $\partial_e(\mathcal{B}_{X^*})$, the set $\{\phi_0\}$ is a weak$^*$-semi-exposed face of $\mathcal{B}_{X^*}$. Then $C$ is a {norm-semi-exposed face} of $\mathcal{B}_X$ if and only if $\Delta(C)$ is a {norm-semi-exposed face} of $\mathcal{B}_Y$.
\end{theorem}

The real interest of the previous theorem is the following corollary.

\begin{corollary}\label{c for spaces with property of semi-exposition for faces}{\rm\cite[Corollary 2.5]{FerGarPeVill17}} Let $X$ and $Y$ be Banach spaces satisfying the following two properties: \begin{enumerate}[$(1)$]\item Every norm closed face of $\mathcal{B}_X$ {\rm(}respectively, of $\mathcal{B}_Y${\rm)} is norm-semi-exposed;
\item Every weak$^*$ closed proper face of $\mathcal{B}_{X^*}$ {\rm(}respectively, of $\mathcal{B}_{Y^*}${\rm)} is weak$^*$-semi-exposed.
\end{enumerate} Let $\Delta: S(X) \to S(Y)$ be a surjective isometry. The following statements hold:
\begin{enumerate}[$(a)$]\item Let $\mathcal{F}$ be a convex set in $S(X)$. Then $\mathcal{F}$ is a norm closed face of $\mathcal{B}_{X}$ if and only if $\Delta(\mathcal{F})$ is a norm closed face of $\mathcal{B}_{Y}$;
\item Let $e\in S(X)$. Then $e\in \partial_e (\mathcal{B}_X)$ if and only if $\Delta(e)\in \partial_e (\mathcal{B}_Y)$.
\end{enumerate}\end{corollary}

As it is observed in \cite{FerGarPeVill17}, the hypotheses of the above corollary hold whenever $X$ and $Y$ are C$^*$-algebras \cite[Theorems 4.10 and 4.11]{AkPed92}, hermitian parts of C$^*$-algebras (see \cite[Corollary 5.1]{EdRutt86} and \cite[Theorem 3.11]{AkPed92}), von Neumann algebra preduals \cite[Theorems 5.3 and 5.4]{EdRutt88}, preduals of the hermitian part of a von Neumann algebra (see \cite[Theorem 4.4]{EdRutt85} and \cite[Theorem 4.1]{EdRutt88}), or more generally, JB$^*$-triples (cf. \cite[Corollary 3.11]{EdFerHosPe2010} and \cite[Corollary 1]{FerPe10}), or JBW$^*$-triple preduals \cite[Corollaries 4.5 and 4.7]{EdRutt88}.\smallskip

By extending a result of D. Tingley \cite[\S 4]{Ting1987}, M. Mori has recently added in \cite[Proposition 2.3]{Mori2017} more information to the conclusion of the above Corollary \ref{c for spaces with property of semi-exposition for faces}. Actually with similar arguments we can deduce the following result.

\begin{proposition}\label{p Mori faces} Let $\Delta: S(X) \to S(Y)$ be a surjective isometry between the unit spheres of two Banach spaces. Then the following statements hold:\begin{enumerate}[$(a)$]\item If $\mathcal{M}$ is a maximal proper face of $\mathcal{B}_X$, then $\Delta(- \mathcal{M})= - \Delta( \mathcal{M});$
\item If $X$ and $Y$ satisfy the hypotheses of Corollary \ref{c for spaces with property of semi-exposition for faces}, then $\Delta(- F)= - \Delta(F)$ for every proper norm closed face of $\mathcal{B}_X$.
\end{enumerate}

\end{proposition}

Elements $a$, $b$ in a C$^*$-algebra $A$ are said to be orthogonal if $a b^* = b^* a =0$. The set of partial isometries in $A$ can be equipped with a partial order defined by $e\leq v$ if $v-e$ is a partial isometry orthogonal to $e$, equivalently, $v= e + (1-ee^*)v(1-v^*v)$. \smallskip

This seems to be an optimal moment to recall the facial structure of the closed unit ball of a C$^*$-algebra. We recall first some basic notions required to understand the results. Let $A$ be a C$^*$-algebra. It was shown by Akemann and Pedersen in \cite{AkPed92} that norm closed faces of $\mathcal{B}_{A}$ are in one-to-one correspondence with the compact partial isometries in $A^{**}$. Let us recall that a projection $p$ in $A^{**}$ is said to be \emph{open} if $A\cap (p A^{**} p)$ is weak$^*$ dense in $p A^{**} p$, equivalently, there exists an increasing net of positive elements in $A$, all of them bounded by $p$, converging to $p$ in the strong$^*$ topology of $A^{**}$ (see \cite[\S 3.11]{Ped}, \cite[\S III.6 and Corollary III.6.20]{Tak}). A projection $p \in A^{**}$ is called \emph{closed} if $1-p$ is open. A closed projection $p$ in $A^{**}$ is called \emph{compact} if $p\leq x$ for some norm-one positive element $x \in A$.\smallskip

Compact partial isometries in the bidual of a C$^{**}$-algebra were studied by C.M. Edwards and G.T. R\"{u}ttimann in \cite[\S 5]{EdRu96} as an application of the more general notion of compact tripotent in the bidual of a JB$^*$-triple. C.A. Akemann and G.K. Pedersen consider an alternative term for the same notion. A partial isometry $v\in A^{**}$ \emph{belongs locally to $A$} if $v^*v$ is a compact projection and there exists a norm-one element $x$ in $A$ satisfying $v = x v^*v$ (compare \cite[Remark 4.7]{AkPed92}). It was shown by C.A. Akemann and G.K. Pedersen that a partial isometry $v$ in $A^{**}$ belongs locally to $A$ if and only if $v^*$ belongs locally to $A$ (see \cite[Lemma 4.8]{AkPed92}). We know from \cite[Theorem 5.1]{EdRu96} that a  partial isometry $v$ in $A^{**}$ belongs locally to $A$ if and only if it is compact in the sense introduced in \cite{EdRu96}.\smallskip

Akemann and Pedersen gave in \cite[Lemma 4.8 and Remark 4.11]{AkPed92} an interesting procedure to understand well those partial isometries in $A^{**}$ belonging locally to $A$. Borrowing a paragraph from the just quoted paper we recall that ``the partial isometries $v$ in $A^{**}$ that belong locally to $A$ are obtained by taking an element $x$ in $A$ with norm 1 and polar decomposition $x = u |x|$ (in $A^{**}$), and then letting $v = ue$ for some compact projection $e$ contained in the spectral projection $\chi_{_{\{1\}}}(|x|)$ of $|x|$ corresponding to the eigenvalue 1.'' Accordingly to most of the basic references, for each element $x$ in $A$ we set $|x| = (x^* x)^{\frac12}$. \smallskip

We are now in position to revisit the results by C.A. Akemann and G.K. Pedersen.

\begin{theorem}\label{t faces AkPed}\cite[Theorems 4.10 and 4.11]{AkPed92} Let $A$ be a C$^*$-algebra. The following statements hold:\begin{enumerate}[$(a)$]\item For each norm closed face $F$ of $\mathcal{B}_A$ there exists a unique partial isometry $v$ in $A^{**}$ belonging locally to $A$ such that $$F= F_v=\{v\}_{_{''}}=\left(v + (1 - vv^*) \mathcal{B}_{A^{**}} (1 - v^*v)\right)\cap \mathcal{B}_{A} = \{x\in \mathcal{B}_A:\ xv^* = vv^*\}.$$ Furthermore, the mapping $v\mapsto F_v$ is an anti-order isomorphism from the complete lattice of partial isometries in $A^{**}$ belonging locally to $A$ onto the complete lattice of norm closed faces of $\mathcal{B}_A$;
\item For each weak$^*$ closed face $\mathcal{G}$ of $\mathcal{B}_{A^{*}}$ there exists a unique partial isometry $v$ in $A^{**}$ belonging locally to $A$ such that $\mathcal{G} = \{v\}_{_{'}}$, and the mapping $v\mapsto \{v\}_{_{'}}$ is an order isomorphism from the complete lattice of partial isometries in $A^{**}$ belonging locally to $A$ onto the complete lattice of weak$^*$ closed faces of $\mathcal{B}_{A^*}$;
\end{enumerate}
\end{theorem}

A non-zero projection $p$ in a C$^*$-algebra $A$ is called \emph{minimal} if $p A p = \mathbb{C} p$. A non-zero partial isometry $e$ in a C$^*$-algebra $A$ is \emph{minimal} if $ee^*$ (equivalently, $e^* e$) is a minimal projection in $A$. By Kadison's transitivity theorem minimal partial isometries in $A^{**}$ belong locally to $A$, and hence every maximal proper face of the unit ball of a C$^*$-algebra $A$ is of the form \begin{equation}\label{eq maximal closed faces} \left(v + (1 - vv^*) \mathcal{B}_{A^{**}} (1 - v^*v)\right)\cap \mathcal{B}_{A}
\end{equation} for a unique minimal partial isometry $v$ in $A^{**}$ (compare \cite[Remark 5.4 and Corollary 5.5]{AkPed92}).\smallskip

Another technical result of geometric nature, which is frequently applied in the study of Tingley's problem and should be considered in any survey on this topic, was established by X.N. Fang, J.H. Wang and G.G. Ding in \cite{FangWang06} and \cite{Ding07}, respectively.

\begin{theorem}\label{t FaWang distance 2}{\rm(\cite[Corollary 2.2]{FangWang06}, \cite[Corollary 1]{Ding07})} Let $X$ and $Y$ be normed spaces and let $\Delta : S(X) \to S(Y)$ be a surjective isometry. Then, for any $x,y$ in $S(X)$, we have $\| x + y \|= 2$ if and only if $\| \Delta(x) + \Delta(y) \|= 2$.
\end{theorem}

This result plays a role, for example, in some of the proofs in \cite{FerGarPeVill17,Per2017}.

\subsection{A taste of Jordan structures}\label{subsec: Jordan}

Many recent advances on Tingley's problem and it's variants on C$^*$-algebras make an explicit use of the Jordan theory of JB$^*$-triples (see, for example, the proofs in \cite{PeTan16,FerPe17b,FerPe17c,FerPe17d} and \cite{FerGarPeVill17}). Although we are not going to enter in the deep details of the proofs, it seems convenient to recall the basic notions and connections with this theory.\smallskip

We recall that, accordingly to the definition introduced in \cite{Ka83}, a \emph{JB$^*$-triple} is a complex Banach space $E$ admitting a continuous triple product $\{a,b,c\}$ which is conjugate linear in $b$ and linear and symmetric in $a$ and $c$, and satisfies the following axioms:\begin{enumerate}[(JB$^*$1)]\item $L(a,b) L(c,d) - L(c,d) L(a,b) = L(L(a,b) (c),d) - L(c,L(b,a)(d))$, for every $a,b,c,d$ in $E$, where $L(a,b)$ is the operator on $E$ defined by $L(a,b) (x) =\{a,b,x\}$;
\item $L(a,a)$ is a hermitian operator on $E$ with non-negative spectrum;
\item $\|\{a,a,a\}\| = \|a \|^3$, for every $a\in E$.
\end{enumerate}

Examples of JB$^*$-triples include the spaces $B(H,H^\prime)$ of bounded linear operators and the spaces $K(H,H^\prime)$ of all compact operators between two complex Hilbert spaces, complex Hilbert spaces, and all C$^*$-algebras when equipped with the triple product defined by $\{x,y,z\} := \frac12 ( xy^* z + z y^*x)$. JB$^*$-triples constitute a category which produces a Jordan model valid to generalize C$^*$-algebras. Every JB$^*$-algebra is a JB$^*$-triple under the triple product $$\{a,b,c\}= (a \circ b^*)\circ c + (c\circ b^*)\circ a - (a\circ c) \circ b^*.$$ For the basic notions and results on JB$^*$-triples the reader is referred to the monograph \cite{Chu2012}.\smallskip

A linear mapping between JB$^*$-triples is called a triple homomorphism if it preserves triple products. Surjective real linear isometries between C$^*$-algebras and JB$^*$-triples are deeply connected to triple isomorphisms (see \cite{Da,ChuDaRuVen} and \cite{FerMarPe}). Many of the results in this survey can be complemented with a good description of the real triple isomorphisms between von Neumann algebras. Let us add that real linear triple isomorphisms play a fundamental role in the original proofs of the main results in \cite{PeTan16, FerPe17b, FerPe17c, FerPe17d}.

\section{Tingley's problem on C$^*$-algebras}\label{sec: Tingleys Cstar}

Tingley's problem for surjective isometries between the unit spheres of two commutative C$^*$-algebras are completely covered by the results for $C_0(L)$-spaces \cite{Wang}, $\ell^\infty (\Gamma)$-spaces \cite{Di:8}, and $L^{p}(\Omega, \Sigma, \mu)$ spaces \cite{Ta:8}. It should be remarked that in \cite{Di:8} and \cite{Ta:8} the authors only consider real sequences and real valued measurable functions, respectively, that is, their results are restricted to the hermitian parts of the corresponding C$^*$-algebras.\smallskip

According to the chronological order, and for our own convenience, we highlight a pioneering result due to R.S. Wang. Let us recall the prototype example of commutative C$^*$-algebras. Given a locally compact Hausdorff space $L$, we shall write $C_0(L)$ for the commutative C$^*$-algebra of all complex valued continuous functions on $L$ which vanish at infinite.

\begin{theorem}\label{t Wang C0(L)}{\rm \cite{Wang}} Let $L_1$ and $L_2$ be two locally compact Hausdorff spaces, and let $\Delta : S(C_0(L_1)) \to S(C_0(L_2))$ be a surjective isometry. Then there exists a real linear surjective isometry $T: C_0(L_1) \to C_0(L_2)$ satisfying $T|_{S(C_0(L_1))} = \Delta$. Furthermore, there exist two disjoint subsets $A$ and $B$ of $L_1$ such that $A\cup B = L_1$, $T|_{C_0(A)}$ is complex linear, and $T|_{C_0(B)}$ is conjugate linear, where $C_0(A)=\{f\in C_0(L_1) : f|_{B} \equiv 0\},$ and $C_0(B)=\{f\in C_0(L_1) : f|_{A} \equiv 0\}.$
\end{theorem}

Wang's theorem, whose proof is based on Urysohn's lemma and fine geometric arguments, solves Tingley's problem for commutative C$^*$-algebras. Actually, if $\Delta : S(\ell_{\infty}(\Gamma_1)) \to S(\ell_{\infty}(\Gamma_2))$ (respectively, $\Delta : S(c(\Gamma_1)) \to S(c(\Gamma_2)),$ or $\Delta : S(c_0(\Gamma_1)) \to S(c_0(\Gamma_2))$) is a surjective isometry, then we can always find an extension to a surjective real linear isometry between the corresponding spaces, where $c_0(\Gamma)$, $c(\Gamma_1),$ and $\ell_{\infty}(\Gamma)$ denote the spaces of all complex null, convergent, and bounded functions on $\Gamma$, respectively. A similar conclusion holds for a surjective isometry $\Delta : S(L^{\infty}(\Omega, \Sigma, \mu)) \to S(L^{\infty}(\Omega, \Sigma, \mu))$.\smallskip

The previous result reveals the importance of considering real linear surjective isometries between $C_0(L)$ spaces. A generalization of the Banach-Stone theorem to real linear surjective isometries (see \cite{Ellis90} and \cite{Miura2011}) assures that for each surjective real linear isometry $T: C_0(L_1) \to C_0(L_2)$ there exist a homeomorphism $\varphi: L_2\to L_1$, a clopen subset $K_2$ of $L_2$, and a unitary continuous function $u :L_2\to \mathbb{C}$ such that $$T(f) (s) = u(s) \ f(\varphi(s)),\ \ \forall f\in C_0(L_1), s\in K_2,$$ and $$T(f) (s) = u(s) \ \overline{f(\varphi(s))},\ \ \forall f\in C_0(L_1), s\in L_2\backslash K_2.$$ Having this theorem in mind, the conclusion in \cite{Di:8} can be explicitly obtained as a consequence of the above Theorem \ref{t Wang C0(L)}.\smallskip

In 2014, 2016, and 2017, R. Tanaka publishes the first achievements on Tingley's problem for surjective isometries between the unit spheres of two non-commutative C$^*$-algebras; his results focus on finite dimensional C$^*$-algebras, and more generally on finite von Neumann algebras (see \cite{Tan2014,Tan2016,Tan2017,Tan2017b}). From now on, we shall write $M_n(\mathbb{C})$ for the space of all $n\times n$ matrices with complex entries.

\begin{theorem}\label{t Tanaka Mn}\cite[Theorem 6.1]{Tan2016} Let $\Delta : S(M_n(\mathbb{C}))\to  S(M_n(\mathbb{C}))$ be a surjective isometry. Then $\Delta$ admits a (unique) extension to a complex linear or to a conjugate linear surjective isometry on $M_n(\mathbb{C})$. Furthermore, there exist a complex linear or conjugate linear $^*$-automorphism $\Phi: M_n(\mathbb{C})\to M_n(\mathbb{C})$ and a unitary matrix $u$ in $M_n(\mathbb{C})$ such that one of the next statements hold:\begin{enumerate}[$(a)$] \item $\Delta (x) = u \Phi(x),$ for all $x\in S(M_n(\mathbb{C}))$;
\item $\Delta (x) = u  \Phi(x)^{*},$ for all $x\in S(M_n(\mathbb{C}))$.
\end{enumerate}
\end{theorem}

Again surjective real linear isometries seem to be behind the results. The proof of the above Theorem \ref{t Tanaka Mn} is based on the following well known fact: The extreme points of the closed unit of $M_n(\mathbb{C})$ are precisely the unitary matrices in $M_n(\mathbb{C})$. Let $\mathcal{U}_n$ denote the set of all unitary matrices in $M_n(\mathbb{C})$. It follows from the above fact and from Corollary \ref{c for spaces with property of semi-exposition for faces} that a surjective isometry $\Delta : S(M_n(\mathbb{C}))\to  S(M_n(\mathbb{C}))$ maps $\mathcal{U}_n$ onto itself, and thus the restriction $\Delta|_{\mathcal{U}_n} : \mathcal{U}_n\to \mathcal{U}_n$ gives a surjective isometry too. Similar conclusions also hold when $M_n(\mathbb{C})$ is replaced by a finite dimensional C$^*$-algebra, or more generally, by a finite von Neumann algebra. We are naturally lead to the an outstanding theorem of O. Hatori and L. Moln{\'a}r.

\begin{theorem}\label{t Hatori Molnar}\cite[Corollary 3]{HatMol2014} Every surjective isometry between the unitary groups of two von Neumann algebras extends to a surjective real linear isometry between the von Neumann algebras. More concretely, let $M_1$ and $M_2$ be von Neumann algebras whose unitary groups are denoted by $\mathcal{U}_1$ and $\mathcal{U}_2$. Let $\Upsilon: \mathcal{U}_1 \to \mathcal{U}_2$ be a bijection. Then $\Upsilon$ is a surjective isometry if and only if there exist a central projection $p\in M_2$ and a Jordan $^*$-isomorphism $\Phi: M_1\to M_2$ such that $$\Upsilon (u) = \Upsilon(1) (p \ \Phi (u) + (1-p) \ \Phi (u)^*),$$ for all $u\in \mathcal{U}_1$.
\end{theorem}

R.V. Kadison and G.K. Pedersen showed in \cite{KadPed} that every element in a finite von Neumann algebra $M$ can be expressed as the convex combination (actually as the midpoint) of two unitary elements in $M$. Tanaka's arguments rely on the facial structure of von Neumann algebras and the property of preservation of midpoints between unitary elements. By this arguments the above Theorem \ref{t Tanaka Mn} was generalized by R. Tanaka in the following form:

\begin{theorem}\label{t Tanaka finite vN}{\rm(\cite[Theorem 4.2]{Tan2017b} and \cite{Tan2017})} Let $\Delta : S(M_1)\to S(M_2)$ be a surjective isometry, where $M_1$ and $M_2$ are finite von Neumann algebras. There exists a surjective real linear isometry $T: M_1\to M_2$ satisfying $\Delta (a) = T(a)$ for all $a\in S(M_1)$. More concretely, we can find a central projection $p\in M_2$ and a Jordan $^*$-isomorphism $\Phi: M_1\to M_2$ such that $$\Delta (a) = \Delta(1) (p \ \Phi (a) + (1-p) \ \Phi (a)^*),$$ for all $a\in S(M_1)$. The same conclusion holds when $\Delta : S(A)\to S(B)$ is a surjective isometry from the unit sphere of a finite dimensional C$^*$-algebra onto the unit sphere of another C$^*$-algebra.
\end{theorem}

The Hatori-Moln{\'a}r theorem is applied by Tanaka to synthesize a surjective real linear isometry $T:M_1\to M_2$.\smallskip

The first results on Tingley's problem for (non-necessarily commutative) operator algebras opened the exploration of this problem for more general clases of operator algebras.\smallskip

The next natural steps are perhaps, the C$^*$-algebras $K(H)$ and $B(H)$ of all compact and bounded linear operators on an infinite dimensional complex Hilbert space $H$, respectively. There is a clear obstruction in the case of $K(H)$ because $\partial_e(\mathcal{B}_{K(H)}) =\emptyset,$ even more, $K(H)$ contains no unitary elements, and hence Theorem \ref{t Hatori Molnar} is meaningless to synthesize a surjective real linear isometry in this setting. Surprisingly, we shall get back to Hatori-Molanr theorem (Theorem \ref{t Hatori Molnar}) when we survey the recent solution to Tingley's problem for general von Neumann algebras obtained in \cite{FerPe17d}.

\subsection{Tingley's problem for compact C$^*$-algebras}\ \smallskip

Along the paper, given a vector $x_0$ in a Banach space $X$, the translation with respect to $x_0$ will be denoted by $\mathcal{T}_{x_0}$.\smallskip

Let us consider the C$^*$-algebra $K(H)$ of all compact operators on an arbitrary complex Hilbert spaces $H$. It is well known that $K(H)^{**} = B(H)$. There is a clear advantage in this case because minimal partial isometries in $K(H)^{**}= B(H)$ are precisely the rank-one partial isometries which clearly belong to $K(H)$. Furthermore, compact partial isometries in $K(H)^{**}$ are all finite rank partial isometries in $K(H)$.\smallskip

A C$^*$-algebra $A$ is called compact if it can be written as a $c_0$-sum of the form $A= \oplus^{c_0}_{j} K(H_j)$, where each $H_j$ is a complex Hilbert space (compare \cite{Alex,Yli}). In this case $A^{**}= \oplus^{\infty}_{j} B(H_j),$ and every minimal partial isometry in $A^{**}$ is a rank-one partial isometry in one of the factors, and hence belongs to $A$. Actually compact partial isometries in $A^{**}$ are finite rank partial isometries, and hence they all belong to $A$. The following proposition was derived in \cite{PeTan16} by combining these facts with Corollary \ref{c for spaces with property of semi-exposition for faces}, the Akemann-Pedersen theorem (see Theorem \ref{t faces AkPed}), the comments in \eqref{eq maximal closed faces}, and Mankiewicz' theorem (see Theorem \ref{t Mankiewicz}).

\begin{proposition}\label{p surjective isometries between the spheres preserve norm closed faces}\cite[Proposition 3.2]{PeTan16} Let $A$ and $B$ be compact C$^*$-algebras, and suppose that $\Delta: S(A) \to S(B)$ is a surjective isometry. Then the following statements hold:\begin{enumerate}[$(a)$] \item $\Delta$ maps norm closed proper faces of $\mathcal{B}_{A}$ to norm closed proper faces of $\mathcal{B}_{B}$;
\item For each {\rm(}minimal{\rm)} partial isometry $e_1$ in $A$ there exists a unique {\rm(}minimal{\rm)} partial isometry $u_1$ in $B$ such that $\Delta\left(\left(e_1 + (1 - e_1 e_1^*) \mathcal{B}_{A^{**}} (1 - e_1^* e_1) \right)\cap \mathcal{B}_{A}\right) = \left(u_1 + (1 - u_1 u_1^*) \mathcal{B}_{B^{**}} (1 - u_1^* u_1)\right)\cap \mathcal{B}_{B}$. Moreover, there exists a surjective real linear isometry $T_{e_1} : (1 - e_1 e_1^*) {A} (1 - e_1^* e_1) \to (1 - u_1 u_1^*) B (1 - u_1^* u_1)$ such that $$\Delta(e_1 + x )= u_1 + T_{e_1} (x),$$ for every $x\in \mathcal{B}_{_{(1 - e_1 e_1^*) {A} (1 - e_1^* e_1)}}$;
\item The restriction of $\Delta$ to each norm closed proper face of $\mathcal{B}_{A}$ is an affine function;
\item For each partial isometry $e_1$ in $A$ there exists a unique partial isometry $u_1$ in $B$ such that $\Delta(e_1 )= u_1$. Moreover, the rank of $e_1$ coincides with the rank of $u_1$ and both are finite.
\end{enumerate}
\end{proposition}

The proof of the above result can be outlined and guessed by the reader from the previously commented results.\smallskip

A result determining when a partial isometry is at distance two from another minimal partial isometry in a compact C$^*$-algebra was first considered in \cite{PeTan16}.

\begin{lemma}\label{l two finite rank tripotents at distance 2 in KH}\cite[Lemma 3.5]{PeTan16} Let $e$ and $w$ be partial isometries in a compact C$^*$-algebra $A$. Suppose that $e$ is minimal and $\|e-w\| = 2$. Then $$w= -e + (1-ee^*) w (1-e^*e).$$
\end{lemma}

Let $\Delta: S(A) \to S(B)$ be a surjective isometry between the unit spheres of two compact C$^*$-algebras. Let us pick a minimal partial isometry $e$ in $A$. Proposition \ref{p surjective isometries between the spheres preserve norm closed faces} implies that $\Delta(e)$ and $\Delta(-e)$ are minimal partial isometries in $B$. Since $\|\Delta(e) - \Delta(-e)\| = \|e+e\| =2$, Lemma \ref{l two finite rank tripotents at distance 2 in KH} assures that $$\Delta (-e) = -\Delta(e) + (1-\Delta(e) \Delta(e)^*) \Delta(-e) (1-\Delta(e)^* \Delta(e)),$$ and we derive from the minimality of $\Delta (-e)$ that $\Delta (-e) = - \Delta (e)$. A more elaborated argument was applied in \cite{PeTan16}, via similar arguments, to establish a version of the original theorem of Tingley \cite{Ting1987} for finite rank partial isometries.

\begin{theorem}\label{t Tingley antipodes for finite rank compact C*-algebras}\cite[Theorem 3.7]{PeTan16} Let $\Delta: S(A) \to S(B)$ be a surjective isometry between the unit spheres of two compact C$^*$-algebras. The following statements hold:\begin{enumerate}[$(a)$]\item If $e$ is a partial isometry in $A$, then $\Delta(-e)=-\Delta(e)$;
\item If $e_1,\ldots,e_m$ are mutually orthogonal partial isometries in $A$, then $\Delta(e_1),$ $\ldots,$ $\Delta(e_m)$ are mutually orthogonal partial isometries in $B$ and $$ \Delta(e_1+\ldots+e_m) = \Delta(e_1)+\ldots+\Delta(e_m).$$
\end{enumerate}
\end{theorem}

If we take a projection $p$ in a C$^*$-algebra $A$, the subspace $(1-p) A (1-p)$ is a C$^*$-subalgebra of $A$. However, if we take a partial isometry $e$ in $A$, the subspace $(1-ee^*) A (1-e^*e)$ need not be, in general, a C$^*$-subalgebra of $A$. However, $(1-ee^*) A (1-e^*e)$ is a norm closed subspace of $A$ which is also closed under the triple product given by \begin{equation}\label{eq C* triple product} \{a,b,c\}=\frac12 (a b^* c + c b^* a).\end{equation} This is equivalent to say that $(1-ee^*) A (1-e^*e)$ is a JB$^*$-subtriple of $A$ is the sense defined in \cite{Ka83} (see subsection \ref{subsec: Jordan}).\smallskip

Suppose that $e$ is a partial isometry in a compact C$^*$-algebra $A$, and let $B$ be another compact C$^*$-algebra. Suppose $\Delta: S(A) \to S(B)$ is a surjective isometry. Let us consider the surjective real linear isometry $$T_e : (1-ee^*) A (1-e^*e) \to (1-\Delta(e)\Delta(e)^*) A (1-\Delta(e)^* \Delta(e))$$ given by Proposition \ref{p surjective isometries between the spheres preserve norm closed faces}$(b)$. Let $e_1$ be any partial isometry in $(1-p) A (1-p)$. By Propositions \ref{p surjective isometries between the spheres preserve norm closed faces} and \ref{t Tingley antipodes for finite rank compact C*-algebras} we have $$\Delta(e) + T_p (e_1) =\Delta (e + e_1 ) = \Delta(e) + \Delta (e_1),$$ we get $T_e (e_1) = \Delta(e_1)$. Furthermore, let $e_1,\ldots,e_m$ be mutually orthogonal partial isometries in $(1-ee^*) A (1-e^*e)$, and let $\alpha_1,\ldots, \alpha_m$ be positive real numbers with $\max\{\alpha_1,\ldots,\alpha_m\}\leq 1$. By the same results above we deduce that $$\Delta\left( e+\sum_{j=1} \alpha_j e_j\right) = \Delta\left( e\right) + T_e \left( \sum_{j=1} \alpha_j e_j\right) = \Delta\left( e\right) + \sum_{j=1} \alpha_j  T_e \left(e_j\right) $$ $$= \Delta\left( e\right) + \sum_{j=1} \alpha_j  \Delta \left(e_j\right).$$ Furthermore, if $w$ is a partial isometry in $A$ such that $e,e_j\in (1- ww^*) A (1-w^*w)$ for all $j$, we also have
\begin{equation}\label{eq on algebraic minimal} \Delta\left( e+\sum_{j=1} \alpha_j e_j\right) =  \Delta\left( e\right) + \sum_{j=1} \alpha_j  \Delta \left(e_j\right)
\end{equation} $$= T_w \left( e\right) + \sum_{j=1} \alpha_j  T_w \left(e_j\right) = T_w\left( e+\sum_{j=1} \alpha_j e_j\right) $$

A triple spectral resolution assures that every compact operator can be approximated in norm by finite linear combinations of mutually orthogonal minimal partial isometries, and the same statement holds for every element in a compact C$^*$-algebra. Therefore, under the above hypotheses, we deduce from the continuity of $T_w$ and $\Delta$ that for each non-zero partial isometry $w\in A$ we have \begin{equation}\label{eq Tw and Delta coincide on the max domain}\hbox{ $\Delta(x) = T_w(x),
$  for all $ x\in S((1- ww^*) A (1-w^*w)).$}
\end{equation} A straight consequence of \eqref{eq Tw and Delta coincide on the max domain} gives the following: if $w_1$ and $w_2$ are non-zero partial isometries we have \begin{equation}\label{eq Tw1 Tw2 and Delta coincide on the intersection}\hbox{ $T_{w_2}(x)=\Delta(x) = T_{w_1}(x),
$}
\end{equation} for all $ x\in S((1- w_{1}w_{1}^*) A (1-w_{1}^*w_{1}))\cap S((1- w_{2}w_{2}^*) A (1-w_{2}^*w_{2})).$\smallskip

The lacking of possibility to apply the Hatori-Moln{\'a}r theorem to synthesize a surjective real linear isometry between $A$ and $B$ forces us to apply a different strategy in \cite{PeTan16}. This different approach is worth to be, at least, outlined here.\smallskip

In a first step we assume that we can find a non-zero subfactor $K(H_1)$ of $A$ such that $A$ is the orthogonal sum of $K(H_1)$ and its orthogonal complement $J = K(H_1)^{\perp}$ and the latter is non-zero.\label{eq proof compact with several subfactors} Let us take two non-zero projections $p_1$ in $K(H_1)$ and $p_2\in J$, and define a mapping $T : A = K(H_1)\oplus^{\perp} J \to B$ given by $$T(x) = T_{p_1} (\pi_2 (x)) + T_{p_{2}} (\pi_1 (x))$$ where $\pi_1$ and $\pi_2$ denote the canonical projections of $A$ onto $K(H_1)$ and $J$, respectively, and $T_{p_1}$ and $T_{p_2}$ are defined by Proposition \ref{p surjective isometries between the spheres preserve norm closed faces}. The mapping $T$ is real linear because $T_{p_1}$ and $T_{p_2}$ are. Clearly $T$ is bounded with $\|T\|\leq 2$. A minimal partial isometry in $A$ either lies in $K(H_1)$ or in $J$. Let us pick an element $x$ in $S(A)$ which can be written in the form $\displaystyle x= e+\sum_{j=1} \alpha_j e_j + \sum_{k=1} \beta_k e_k$, where $e,e_j,e_k$ are mutually ortogonal minimal partial isometries in $A$, $\alpha_j$, $\beta_k\in \mathbb{R}^+,$ $e_j\in B(H_1)$ and $e_k\in J$ for all $j$, $k,$ and $e$ either lies in $B(H_1)$ or in $J$. If $e\in K(H_1)$ (respectively, $e\in J$), by \eqref{eq Tw and Delta coincide on the max domain}, we have $\Delta(e) = T_{p_1} (e)=T(e)$ (respectively, $\Delta(e) = T_{p_2} (e)=T(e)$). Now, by \eqref{eq on algebraic minimal} and \eqref{eq Tw and Delta coincide on the max domain} we have $$ \Delta(x) = \Delta( e) +\sum_{j=1} \alpha_j \Delta(e_j) + \sum_{k=1} \beta_k \Delta(e_k) = \Delta( e) +\sum_{j=1} \alpha_j T_{p_2}(e_j) + \sum_{k=1} \beta_k T_{p_1} (e_k)$$ $$ = T( e) +\sum_{j=1} \alpha_j T(e_j) + \sum_{k=1} \beta_k T(e_k) = T(x).$$ The norm density of this kind of elements $x$ in $S(A)$ together with the norm continuity of $T$ and $\Delta$ proves that $T(x) = \Delta(x)$ for all $x\in S(A)$.\smallskip

In the second case we assume that $A= K(H)$ for some complex Hilbert space $H$. If $H$ is finite dimensional Theorem \ref{t Tanaka finite vN} proves that our mapping $\Delta: S(A)\to S(B)$ admits a unique extension to a surjective real linear isometry from $A$ onto $B$. We can therefore assume that $H$ is infinite dimensional. \smallskip

Let us take three mutually orthogonal minimal projections $p_1,p_2$ and $p_3$ in $A$, and the corresponding surjective real linear isometries $T_{p_1}$, $T_{p_2},$ and $T_{p_3}$ given by Proposition \ref{p surjective isometries between the spheres preserve norm closed faces}. We can decompose $A$ in the form $$ A = \mathbb{C} p_1 \oplus (p_1 A p_2\oplus p_2 A p_1)\oplus ((1-p_2) A p_1\oplus p_1 A (1-p_2)) \oplus (1-p_1) A (1-p_1),
$$ where $\mathbb{C} p_1 \oplus (p_1 A p_2\oplus p_2 A p_1)\subset (1-p_3) A (1-p_3),$ and $((1-p_2) A p_1\oplus p_1 A (1-p_2))\subset (1-p_2) A (1-p_2).$ Let $\pi_1$, $\pi_2,$ and $\pi_3$ denote the corresponding projections of $A$ onto $\mathbb{C} p_1 \oplus (p_1 A p_2\oplus p_2 A p_1)$, $((1-p_2) A p_1\oplus p_1 A (1-p_2))$ and $(1-p_1) A (1-p_1)$, respectively. We synthesize\label{label synthesis on K(H)} a mapping $T: A\to B$ given by $$T(x) =   T_{p_3} (\pi_1 ( x)) +  T_{p_2}(\pi_2 (x))+ T_{p_1}(\pi_3(x)).$$ The mapping $T$ is continuous and real linear because $T_{p_1}$, $T_{p_2}$ and $T_{p_3}$ are.\smallskip

If we prove that $$ T(e) =\Delta(e), \hbox{ for every minimal partial isometry $e$ in $A$},$$ then a similar argument to that given in the first step above, based on \eqref{eq on algebraic minimal} and \eqref{eq Tw and Delta coincide on the max domain}, the norm density in $S(A)$ of elements which can be written as finite positive combinations of mutually orthogonal projections, and the continuity of $T$ and $\Delta$, shows that $T(x) = \Delta(x)$ for all $x\in S(A)$.\smallskip

Let $e$ be a minimal partial isometry in $A$. Since $H$ is infinite dimensional, we can find another minimal projection $p_4$ which is orthogonal to $p_1, p_2, p_3, e$.\smallskip

Since $e\in (1-p_4) A (1-p_4)$, the statement in \eqref{eq Tw and Delta coincide on the max domain} implies that $\Delta (e) = T_{p_4} (e).$\smallskip

Let us write $e= p_1 e p_1 + p_1 e p_{2}+ p_2 e p_1+ p_1 e (1-p_{2}) + (1-p_2) e p_1+ (1-p_1) e(1-p_1)$. Clearly,  $p_1 e p_1, p_1 e p_{2}, p_2 e p_1\in (1-p_4) A (1-p_4)$. Since $p_1,p_2, e\in (1-p_4) A (1-p_4)$, we also deduce that $p_1 e (1-p_{2}), (1-p_2) e p_1, (1-p_1) e(1-p_1) \in (1-p_4) A (1-p_4).$ By applying \eqref{eq Tw1 Tw2 and Delta coincide on the intersection} to $T_{p_4}$ and $T_{p_3}$ (respectively, to $T_{p_4}$ and $T_{p_2}$, and $T_{p_4}$ and $T_{p_3}$) we get $$T(e) = T_{p_3} (p_1 e p_1 + p_1 e p_{2}+ p_2 e p_1) +  T_{p_2}(p_1 e (1-p_{2}) + (1-p_2) e p_1)+ T_{p_1}((1-p_1) e(1-p_1)) $$ $$= T_{p_4} (p_1 e p_1 + p_1 e p_{2}+ p_2 e p_1) +  T_{p_4}(p_1 e (1-p_{2}) + (1-p_2) e p_1)+ T_{p_4}((1-p_1) e(1-p_1)) $$ $$= T_{e_4} (e) = \Delta(e).$$

We have sketched the main arguments leading to one of the main achievements in \cite{PeTan16}.

\begin{theorem}\label{thm Tingley compact Cstaralgebras}\cite[Theorem 3.14]{PeTan16} Let $\Delta: S(A)\to S(B)$ be a surjective isometry between the unit spheres of two compact C$^*$-algebras. Then there exists a (unique) surjective real linear isometry $T:A\to B$ such that $T(x) = \Delta(x),$ for every $x$ in $S(A)$. In particular, the same conclusion holds when $A=K(H_1)$ and $B=K(H_2)$, where $H_1$ and $H_2$ are arbitrary complex Hilbert spaces.
\end{theorem}

Surjective real linear isometries between (real) C$^*$-algebras were studied in deep by Ch.H. Chu, T. Dang, B. Russo, B. Ventura in \cite{ChuDaRuVen}. Theorem 6.4 in \cite{ChuDaRuVen} proves that every surjective real linear isometry between (real) C$^*$-algebras is a triple isomorphism with respect to the triple product defined in \eqref{eq C* triple product}. Studies on surjective real linear isometries on JB$^*$-triples and real JB$^*$-triples have been considered by T. Dang \cite{Da} and F.J. Fern{\'a}ndez-Polo, J. Mart{\' i}nez and the author os this survey \cite{FerMarPe}.

\subsection{Tingley's problem for $B(H)$}\ \smallskip

After having revisited the solution to Tingley's problem for compact C$^*$-algebras published in \cite{PeTan16}, the next natural challenge is to consider a surjective isometry $\Delta: S(B(H_1)) \to S(B(H_2)$, where $H_1,H_2$ are arbitrary complex Hilbert spaces. Let us observe that if $H_1$ or $H_2$ is finite dimensional, then the extension of $\Delta$ to a surjective real linear isometry is guaranteed by Tanaka's theorem (see Theorem \ref{t Tanaka finite vN}).\smallskip

The problem in the setting of $B(H)$ spaces has been recently solved in a contribution by F.J. Fern{\'a}ndez-Polo and the author of this survey in \cite{FerPe17b}. The main conclusion gives a positive solution to Tingley's problem in the setting just commented.

\begin{theorem}\label{t Tingley BH spaces}\cite[Theorem 3.2]{FerPe17b} Let $H_1$ and $H_2$ be complex Hilbert spaces. Suppose that $\Delta: S(B(H_1)) \to S(B(H_2))$ is a surjective isometry. Then there exists a surjective complex linear or conjugate linear surjective isometry $T$ from $B(H_1)$ onto $B(H_2)$ satisfying $\Delta(x) = T(x)$, for every $x\in S(B(K))$.
\end{theorem}

Actually, a stronger conclusion has been achieved.

\begin{theorem}\label{thm Tyngley ellinfty sums}\cite[Theorem 3.2]{FerPe17b} Let $(H_i)_{i\in I}$ and $(K_j)_{j\in J}$ be two families of complex Hilbert spaces. Suppose $\Delta: S\left(\bigoplus_j^{\ell_{\infty}} B(K_j) \right) \to S\left(\bigoplus_i^{\ell_{\infty}} B(H_i) \right)$ is a surjective isometry. Then there exists a surjective real linear isometry $$T: S\left(\bigoplus_j^{\ell_{\infty}} B(K_j) \right) \to S\left(\bigoplus_i^{\ell_{\infty}} B(H_i) \right)$$ satisfying $T|_{S(E)} = \Delta$.
\end{theorem}

The strategy to obtain the previous two theorems also begins with results based on the facial structure of the closed unit ball of $B(H)$, Theorem \ref{t faces ChengDong11}, Corollary \ref{c for spaces with property of semi-exposition for faces} and the Akemann-Pedersen theorem (Theorem \ref{t faces AkPed}). The latter result forces us to face a serious additional obstacle which requires a completely new strategy. More concretely, we have already seen in the previous subsection that, for a compact C$^*$-algebra $A$, the norm closed faces of $\mathcal{B}_{A}$ are determined by finite rank partial isometries in $A$. However, for a general C$^*$-algebra $A$ the maximal proper faces of $\mathcal{B}_{A}$ are determined by minimal partial isometries in $A^{**}$. This is a serious obstacle which makes invalid the arguments in previous subsections and in \cite{PeTan16,FerPe17} to the case of a surjective isometry $\Delta: S(B(H_1)) \to S(B(H_2))$, because, in principle, $\Delta$ cannot be applied to every minimal projection in $B(H_1)^{**}$. The novelties in \cite{FerPe17b} are based on certain technical results which provides an antidote to avoid these difficulties.\smallskip

Two results from \cite{FerPe17b} deserve to be highlighted by their own right.

\begin{theorem}\label{t surjective isometries map minimal partial isometries into points of strong subdiff}\cite[Theorem 2.3]{FerPe17b}
Let $A$ and $B$ be C$^*$-algebras, and suppose that $\Delta: S(A) \to S(B)$ is a surjective isometry. Let $e$ be a minimal partial isometry in $A$. Then $1$ is isolated in the spectrum of $|\Delta(e)|$.
\end{theorem}

The consequences of the previous result reveal to be stronger after the next additional theorem.

\begin{theorem}\label{t surjective isometries map minimal partial isometries into minimal partial isometries}\cite[Theorem 2.5]{FerPe17b}
Let $A$ be a C$^*$-algebra, and let $H$ be a complex Hilbert space. Suppose that $\Delta: S(A) \to S(B(H))$ is a surjective isometry. Let $e$ be a minimal partial isometry in $A$. Then $\Delta (e)$ is a minimal partial isometry in $B(H)$. Moreover, there exists a surjective real linear isometry $$T_{e} : (1 - ee^*) A  (1 - e^*e)\to  {(1 - \Delta(e)\Delta(e)^*) B(H) (1 - \Delta(e)^*\Delta(e))}$$ such that $$ \Delta(e + x) = \Delta(e) + T_e (x), \hbox{ for all $x$ in } \mathcal{B}_{(1 - ee^*) A  (1 - e^*e)}.$$ In particular, the restriction of $\Delta$ to the face $F_{e} =e + (1 - ee^*) \mathcal{B}_A (1 - e^*e)$ is a real affine function.
\end{theorem}

Technical algebraic and geometric manipulations combined with the previous theorem determine a precise control of a surjective isometry $\Delta: S(B(K)) \to S(B(H))$ on algebraic elements in the sphere which can be expressed as finite positive linear combinations of mutually orthogonal minimal partial isometries. It should be remarked here that a traditional spectral resolution with finite linear combinations of mutually orthogonal projections is only valid to approximate hermitian elements in the sphere.

\begin{theorem}\label{t surjective isometries between spheres of B(H)}\cite[Theorem 2.7]{FerPe17b} Let $\Delta: S(B(H_1)) \to S(B(H_2))$ be a surjective isometry where $H_1$ and $H_2$ are complex Hilbert spaces with dimension greater than or equal to 3. Then the following statements hold:\begin{enumerate}[$(a)$]\item For each minimal partial isometry $v$ in $B(H_1)$, the mapping $$T_v : (1 - vv^*) B(H_1)  (1 - vv^*)\to  {(1 - \Delta(v)\Delta(v)^*) B(H_2) (1 - \Delta(v)^*\Delta(v))}$$ given by Theorem \ref{t surjective isometries map minimal partial isometries into minimal partial isometries} is complex linear or conjugate linear;
\item For each minimal partial isometry $v$ in $B(H_1)$ we have $\Delta(-v) = -\Delta(v)$ and $T_{v} = T_{-v}$. Furthermore, $T_v$ is weak$^*$-continuous and  $\Delta(e) = T_v(e)$ for every minimal partial isometry $e\in (1 - vv^*) B(H_1)  (1 - v^* v)$;
\item For each minimal partial isometry $v$ in $B(H_1)$ the equality $\Delta(w) = T_v (w)$ holds for every partial isometry  $w\in (1 - vv^*) B(H_1)  (1 - v^*v)\backslash\{0\}$;
\item Let $w_1,\ldots,w_n$ be mutually orthogonal non-zero partial isometries in $B(H_1)$, and let $\lambda_1,\ldots,\lambda_n$ be positive real numbers with $\lambda_1=1,$ and $\lambda_j\leq 1$ for all $j$. Then $$\Delta\left(\sum_{j=1}^n \lambda_j w_j\right) = \sum_{j=1}^n \lambda_j \Delta\left(w_j\right);$$
\item For each minimal partial isometry $v$ in $B(H_1)$ we have $\Delta(x) = T_v (x)$ for every $x\in S(\mathcal{B}_{(1 - vv^*) B(H_1)  (1 - v^* v)})$;
\item For each partial isometry $w$ in $B(H_1)$ the element $\Delta(w)$ is a partial isometry;
\item Suppose $v_1,v_2$ are mutually orthogonal minimal partial isometries in $B(H_1)$ then $T_{v_1} (x) = T_{v_2} (x)$ for every $x$ in the intersection $$\left( (1 - v_1v_1^*) B(H_1)  (1 - v_1v_1^*)\right) \cap \left((1 - v_2v_2^*) B(H_1)  (1 - v_2v_2^*)\right);$$
\item Suppose $v_1,v_2$ are mutually orthogonal minimal partial isometries in $B(H_1)$ then exactly one of the following statements holds:\begin{enumerate}[$(1)$]\item The mappings $T_{v_1}$ and $T_{v_2}$ are complex linear;
\item The mappings $T_{v_1}$ and $T_{v_2}$ are conjugate linear.
\end{enumerate}
\end{enumerate}
\end{theorem}

The synthesis of a surjective real linear isometry in the proof of Theorem \ref{t Tingley BH spaces} (see \cite[Theorem 3.2]{FerPe17b}) is given with similar arguments to those we sketched in page \pageref{label synthesis on K(H)} with the obvious modifications and the new tools developed in Theorems \ref{t surjective isometries map minimal partial isometries into minimal partial isometries} and \ref{t surjective isometries between spheres of B(H)}. That is, assuming that $H$ is infinite dimensional, we pick three mutually orthogonal minimal projections $p_1,p_2$ and $p_3$ in $A$, and the corresponding surjective real linear isometries $T_{p_1}$, $T_{p_2},$ and $T_{p_3}$ given by Theorem \ref{t surjective isometries map minimal partial isometries into minimal partial isometries}. By decomposing $B(H_1)$ in the form $$B(H_1) = \mathbb{C} p_1 \oplus (p_1 B(H_1) p_2\oplus p_2 B(H_1) p_1) $$ $$\oplus ((1-p_2) B(H_1) p_1\oplus p_1 B(H_1) (1-p_2)) \oplus (1-p_1) B(H_1) (1-p_1),$$ with $$\mathbb{C} p_1 \oplus (p_1 B(H_1) p_2\oplus p_2 B(H_1) p_1)\subset (1-p_3) B(H_1) (1-p_3),$$ and $$((1-p_2) B(H_1) p_1\oplus p_1 B(H_1) (1-p_2))\subset (1-p_2) B(H_1) (1-p_2),$$ and denoting by $\pi_1$, $\pi_2,$ and $\pi_3$ the corresponding projections of $B(H_1)$ onto $\mathbb{C} p_1 \oplus (p_1 B(H_1) p_2\oplus p_2 B(H_1) p_1)$, $((1-p_2) B(H_1) p_1\oplus p_1 B(H_1) (1-p_2))$ and $(1-p_1) B(H_1) (1-p_1)$, respectively. We synthesize a mapping $T: B(H_1)\to B(H_2)$ given by $$T(x) =   T_{p_3} (\pi_1 ( x)) +  T_{p_2}(\pi_2 (x))+ T_{p_1}(\pi_3(x)).$$ The mapping $T$ is weak$^*$ continuous and real linear because $T_{p_1}$, $T_{p_2}$ and $T_{p_3}$ are. By the new tools given by Theorem \ref{t surjective isometries between spheres of B(H)} it is shown in the proof of \cite[Theorem 3.2]{FerPe17b} that $\Delta (e) = T(e)$ for every minimal partial isometry $e$ in $B(H_1)$.\smallskip

Contrary to the case of $K(H)$ spaces and compact C$^*$-algebras, where every element in the sphere can be approximated in norm by norm-one elements which are finite linear combination of mutually orthogonal minimal partial isometries, elements in the sphere of $B(H)$ can be approximated only in the weak$^*$ topology by these kind of algebraic elements. To solve this additional obstacle, it is established in \cite{FerPe17b} an identity principle in the following terms.

\begin{proposition}\label{p surjective isometries extension on projections}\cite[Proposition 3.1]{FerPe17b} Let $H_1$ and $H_2$ be complex Hilbert spaces. Suppose that $\Delta: S(B(H_1)) \to S(B(H_2))$ is a surjective isometry, and there exists a weak$^*$-continuous real linear operator  $T:B(H_1)\to B(H_2)$ such that $\Delta(v)=T(v)$, for every minimal partial isometry $v$ in $B(H_1)$. Then $T$ and $\Delta$ coincide on the whole $S(B(H_1))$.
\end{proposition}

The above proposition, Theorem \ref{t surjective isometries between spheres of B(H)} and the above observation are, in essence, all the arguments required to prove Theorem \ref{t Tingley BH spaces}. The proof of Theorem \ref{thm Tyngley ellinfty sums} required additional technical adaptations which can be found in \cite{FerPe17b}.

\subsection{Tingley's problem for von Neumann algebras} \ \smallskip

The most recent, and for the moment, the most general conclusion on Tingley's problem is an affirmative solution to this problem for surjective isometries between the unit spheres of two arbitrary von Neumann algebras, which has been recently obtained by F.J. Fern{\'a}ndez-Polo and the author of this survey in \cite{FerPe17d}. The result reads as follows:

\begin{theorem}\label{t Tingley von Neumann}\cite[Theorem 3.3]{FerPe17d} Let $\Delta : S(M)\to S(N)$ be a surjective isometry between the unit spheres of two von Neumann algebras. Then there exists a surjective real linear isometry $T: M\to N$ whose restriction to $S(M)$ is $\Delta$. More precisely, there exist a central projection $p$ in $N$ and a Jordan $^*$-isomorphism $J : M\to N$ such that, defining $T : M\to N$ by $T(x) = \Delta(1) \left( p J (x) + (1-p) J (x)^*\right)$ {\rm(}$x\in M${\rm)}, then $T$ is a surjective real linear isometry and $T|_{S(M)} = \Delta$.
\end{theorem}

The mathematical difficulties of the problem in this general setting are considerable. The techniques, procedures and strategies applied in previous case to synthesize a surjective real linear isometry and to apply the facial structure are no longer valid under the new hypotheses.\smallskip

Let $\Delta: S(A)\to S(B)$ be a surjective isometry between the unit spheres of two C$^*$-algebras. A combination of Theorem \ref{t faces AkPed} and Corollary \ref{c for spaces with property of semi-exposition for faces} (see also the subsequent comments) gives a one-to-one correspondence between compact partial isometries in the corresponding second duals.

\begin{theorem}\label{t first correspondence between faces and compact partial isometries in the bidual for a surjective isometry} Let $\Delta: S(A) \to S(B)$ be a surjective isometry between the unit spheres of two C$^*$-algebras. Then the following statements hold:
\begin{enumerate}[$(a)$]\item For each non-zero compact partial isometry $e\in A^{**}$ there exists a unique (non-zero) compact partial isometry $\phi_{\Delta} (e)\in B^{**}$ such that $\Delta (F_e) = F_{\phi_\Delta (e)},$ where $F_{e} =\left(e + (1 - ee^*) \mathcal{B}_{A^{**}} (1 - e^*e)\right)\cap \mathcal{B}_A$;
\item The mapping $e\mapsto \phi_\Delta (e)$ defines an order preserving bijection between the sets of non-zero compact partial isometries in $A^{**}$ and the set of non-zero compact partial isometries in $B^{**}$;
\item $\phi_{\Delta}$ maps minimal partial isometries in $A^{**}$ to minimal partial isometries in $B^{**}$.
\end{enumerate}
\end{theorem}

The above result produces no alternative to our obstacles because compact partial isometries in the second dual cannot be transformed under $\Delta$.  Technical arguments based on ultraproducts techniques and a subtle uniform generalization of Lemma \ref{l two finite rank tripotents at distance 2 in KH} are appropriately applied in \cite{FerPe17d} to obtain generalizations of the above Theorems \ref{t surjective isometries map minimal partial isometries into points of strong subdiff} and \ref{t surjective isometries map minimal partial isometries into minimal partial isometries}.

\begin{theorem}\label{t surjective isometries map partial isometries into points of strong subdiff}\cite[Theorem 2.7]{FerPe17d}
Let $\Delta: S(A) \to S(B)$ be a surjective isometry between the unit spheres of two C$^*$-algebras. Let $e$ be a non-zero partial isometry in $A$. Then $1$ is isolated in the spectrum of $|\Delta(e)|$.
\end{theorem}

Mankiewicz's theorem (Theorem \ref{t Mankiewicz}) plays a fundamental role in the second part of the statement of the next theorem.

\begin{theorem}\label{t A}\cite[Theorem 2.8]{FerPe17d}
Let $\Delta: S(A) \to S(B)$ be a surjective isometry between the unit spheres of two C$^*$-algebras. Then $\Delta$ maps non-zero partial isometries in $A$ to non-zero partial isometries in $B$. Moreover, for each non-zero partial isometry $e$ in $A$, we have $\phi_{\Delta} (e) = \Delta(e)$, and  there exists  a surjective real linear isometry $$T_{e} : (1 - ee^*) A  (1 - e^*e)\to  {(1 - \Delta(e)\Delta(e)^*) B (1 - \Delta(e)^*\Delta(e))}$$ such that $$ \Delta(e + x) = \Delta(e) + T_e (x), \hbox{ for all $x$ in } \mathcal{B}_{(1 - ee^*) A  (1 - e^*e)}.$$ In particular the restriction of $\Delta$ to the face $F_{e} =e + (1 - ee^*) \mathcal{B}_A (1 - e^*e)$ is a real affine function.
\end{theorem}

Another crucial step in the study of Tingley's problem on von Neumann algebras asserts that the mapping $\phi_\Delta$ given by Theorem \ref{t first correspondence between faces and compact partial isometries in the bidual for a surjective isometry} preserves antipodal points.

\begin{theorem}\label{t C}\cite[Theorem 2.11]{FerPe17d} Let $\Delta: S(A) \to S(B)$ be a surjective isometry between the unit spheres of two C$^*$-algebras. Then, for each non-zero compact partial isometry $e$ in $A^{**}$ we have $\phi_\Delta (-e) = -\phi_\Delta (e)$, where $\phi_\Delta$ is the mapping given by Theorem \ref{t first correspondence between faces and compact partial isometries in the bidual for a surjective isometry}. Consequently, for each non-zero partial isometry $e\in A$ we have $\Delta(-e) = -\Delta(e)$.
\end{theorem}

The ortogonal complement of a subset $S$ in a C$^*$-algebra $A$ is defined by $$S^{\perp} :=\{x\in A : x\perp a, \hbox{ for all } a\in S\}.$$

The previous theorems provide the key tools to extend Theorem \ref{t surjective isometries between spheres of B(H)} to te setting of von Neumann algebras.

\begin{proposition}\label{p algebraic elements}\cite[Proposition 2.12]{FerPe17d} Let $\Delta: S(A) \to S(B)$ be a surjective isometry between the unit spheres of two C$^*$-algebras. Then the following statements hold:\begin{enumerate}[$(a)$]\item For each non-zero partial isometry $v$ in $A$, the surjective real linear isometry $$T_v : (1 - vv^*) A  (1 - vv^*)\to  {(1 - \Delta(v)\Delta(v)^*) B (1 - \Delta(v)^*\Delta(v))}$$ given by Theorem \ref{t A} satisfies $\Delta (e) = T_v(e),$ for every non-zero partial isometry $e\in (1 - vv^*) A  (1 - v^* v)$;
\item Let $w_1,\ldots,w_n$ be mutually orthogonal non-zero partial isometries in $A$, and let $\lambda_1,\ldots,\lambda_n$ be real numbers with $1=|\lambda_1|\geq\max\{|\lambda_j|\}$. Then $$\Delta\left(\sum_{j=1}^n \lambda_j w_j\right) = \sum_{j=1}^n \lambda_j \Delta\left(w_j\right);$$
\item Suppose $v,w$ are mutually orthogonal non-zero partial isometries in $A$ then $T_{v} (x) = T_{w} (x)$ for every $x\in \{v\}^{\perp} \cap \{w\}^{\perp}$;
\item If $A$ is a von Neumann algebra, then for each non-zero partial isometry $v$ in $A$ we have $\Delta(x) = T_v (x)$ for every $x\in S({(1 - vv^*) A  (1 - v^* v)}).$
\end{enumerate}
\end{proposition}

Given a surjective isometry $\Delta: S(M) \to S(N)$ between the unit spheres of two von Neumann algebras, the synthesis of a surjective real linear extension to a surjective real linear isometry $T:M\to N$ follows completely different arguments than those in the cases of compact C$^*$-algebras and $B(H)$. The technique in this case relies again in the Hatori-Moln{\'a}r theorem (Theorem \ref{t Hatori Molnar}). R. Tanaka proves in \cite{Tan2017b} that a surjective isometry between the unit spheres of two finite von Neumann algebras maps unitary elements to unitary elements. This result has been extended to general von Neumann algebras in \cite[Theorem 3.2]{FerPe17d}.

\begin{theorem}\label{t Tanaka unitaries general vN}\cite[Theorem 3.2]{FerPe17d} Let $\Delta : S(M)\to S(N)$ be a surjective isometry between the unit spheres of two von Neumann algebras.  Then $\Delta$ maps unitaries in $M$ to unitaries in $N$.
\end{theorem}

From now on, let the symbol $\mathcal{U} (A)$ denote the unitary group of a C$^*$-algebra $A$. The above Theorem \ref{t Tanaka unitaries general vN} opens the door to apply the Hatori-Moln{\'a}r theorem (Theorem \ref{t Hatori Molnar}) to synthesize a surjective real linear isometry $T : M \to N$ satisfying that $T(u) = \Delta (u)$ for all $u\in \mathcal{U} (M)$. The difficulties to finish the proof of Theorem \ref{t Tingley von Neumann} reside in proving that $\Delta(x) = T(x)$ for all $x\in S(M)$. This is solved in \cite{FerPe17d} with a convenient application of the theory of convex combinations of unitary operators in von Neumann algebras developed by C.L. Olsen and G.K. Pedersen in \cite{Ols1989} and \cite{OlsPed1986}. These are the main lines in the proof of Theorem \ref{t Tingley von Neumann}.\smallskip

It is worth to make a stop to comment the first connection with a very recent contribution of M. Mori. In the preprint \cite{Mori2017}, M. Mori establishes a generalization of the above Theorem \ref{t Tanaka unitaries general vN}.

\begin{theorem}\label{t Tanaka unitaries general unital Cstar}\cite[Theorem 3.2]{Mori2017} Let $\Delta : S(A)\to S(B)$ be a surjective isometry between the unit spheres of two unital C$^*$-algebras. Then $\Delta$ maps unitaries in $A$ to unitaries in $B$.
\end{theorem}

The proof presented by M. Mori in \cite{Mori2017} is based in the following geometric result, which is a nice discovering by itself, and might be useful in some other contexts.

\begin{lemma}\label{l Mori 3.1} Let $A$ be a unital C$^*$-algebra, and let $x$ be an element in $\partial_e(\mathcal{B}_A)$. Then $x$ is a unitary if and only if the set $A_x :=\{ y \in \partial_e(\mathcal{B}_A) : \|x\pm y\| =\sqrt{2} \}$ has an isolated point as metric space.
\end{lemma}

We finish this section with a couple of open problems.

\begin{open problem}\label{open problem Tingley}{\rm(}Tingley's problem for C$^*$-algebras{\rm)} Let $\Delta : S(A)\to S(B)$ be a surjective isometry between the unit spheres of two C$^*$-algebras. Does $\Delta$ admits an extension to a surjective real linear isometry from $A$ onto $B$?
\end{open problem}

When $A$ is unital C$^*$-algebra M. Mori shows in \cite[Proposition 3.4 and Problem 6.1]{Mori2017} that a particular version of the above problem can be restated in the following terms:

\begin{open problem} Let $A$ be a unital C$^*$-algebra and let $\Delta : S(A)\to S(A)$ be a surjective isometry. Suppose that $\Delta(x) = x$ for every invertible element in the unit sphere of $A$. Is $\Delta$ equal to the identity mapping on $S(A)$?
\end{open problem}

\section{Tingley's problem on von Neumann algebra preduals}\label{sec: Tingleys predual}

Let us begin this section with another result due to G.G. Ding. Let $\Gamma$ be a index set, we denote by $\ell_{_{\mathbb{R}}}^1(\Gamma)$ the Banach space of all absolutely summable families of real numbers equipped with the norm $\displaystyle \left\| (\xi_j)_j \right\|_1 = \sum_{j\in \Gamma} |\xi_j|.$

\begin{theorem}\label{t Ding ell1}\cite[Theorem 1]{Di:1} Let $\Delta : S(\ell_{_{\mathbb{R}}}^1(\Gamma_1)) \to S(\ell_{_{\mathbb{R}}}^1(\Gamma_2))$ be a surjective isometry. Then there exists a one-to-one bijection $\sigma : \Gamma_1\to \Gamma_2$ and a family of real numbers $\{\theta_j : j \in \Gamma_1\}\subseteq \mathbb{T}$ such that $$\Delta\left( \sum_{j\in \Gamma_1} \xi_j e_j \right) = \sum_{j\in \Gamma_2} \theta_j \xi_{\sigma(j)} \widehat{e}_j,$$ where $\{e_j :j\in \Gamma_1\}$ and $\{\widehat{e}_j :j\in \Gamma_2\}$ are the canonical basis of $\ell_{_{\mathbb{R}}}^1(\Gamma_1)$ and $\ell_{_{\mathbb{R}}}^1(\Gamma_2)$, respectively. In particular, there exists a surjective real linear isometry $T : \ell_{_{\mathbb{R}}}^1(\Gamma_1) \to \ell_{_{\mathbb{R}}}^1(\Gamma_2)$ whose restriction to $S(\ell_{_{\mathbb{R}}}^1(\Gamma_1))$ coincides with $\Delta$.
\end{theorem}

Given a $\sigma$-finite measure space $(\Omega, \Sigma, \mu)$, the symbol $L_{_\mathbb{R}}^{1}(\Omega, \Sigma, \mu)$ will denote the Banach space of real valued measurable functions $f: \Omega\to \mathbb{R}$ satisfying $\displaystyle \int_{\Omega} |f| d\mu <\infty,$ with norm $\|f\|_1 = \displaystyle \int_{\Omega} |f| d\mu$. The Banach space of all essentially bounded real valued measurable functions on $\Omega$ will be denoted by $L_{_\mathbb{R}}^{\infty}(\Omega, \Sigma, \mu)$.\smallskip

The previous result of Ding is complemented by the following result due to D. Tan.

\begin{theorem}\label{t Tan L1}\cite[Theorem 3.4]{Ta:1} Let $(\Omega, \Sigma, \mu)$ be a $\sigma$-finite measure space and let $Y$ be a real Banach space. Then every surjective isometry $\Delta : S(L_{_\mathbb{R}}^{1}(\Omega, \Sigma, \mu)) \to S(Y)$ can be uniquely extended to a surjective real linear isometry from $L_{_\mathbb{R}}^{1}((\Omega, \Sigma, \mu))$ onto $Y$.
\end{theorem}

Regarding $\ell_{_{\mathbb{R}}}^1(\Gamma_1)$ and $L_{_\mathbb{R}}^{1}(\Omega, \Sigma, \mu)$ as predual spaces of the hermitian parts of the von Neumann algebras $\ell_{_{\mathbb{R}}}^\infty(\Gamma_1)$ and $L_{_\mathbb{R}}^{\infty}(\Omega, \Sigma, \mu)$, respectively, it seems natural to ask whether Theorems \ref{t Ding ell1} and \ref{t Tan L1} admits non-commutative counterparts. The duality $c_0^*=\ell^1$ and $(\ell^1)^*= \ell^\infty$ admits a non-commutative alter ego in the form $K(H)^*=C_1(H)$ and $C_1(H) = B(H)$, where $C_1(H)$ is the space of trace class operators on a complex Hilbert space $H$. This will be treated in the next subsection.

\subsection{Tingley's problem on trace class operators}\ \smallskip

Tingley's problem for surjective isometries between unit spheres of spaces of trace class operators has been approached by F.J. Fern{\'a}ndez-Polo, J.J. Garc{\'e}s, I. Villanueva and the author of this note in \cite{FerGarPeVill17}. We shall review here the main achievements in this line.  \smallskip

When the space $C_1(H)$ is regarded as the predual of the von Neumann algebra $B(H),$ or as the dual space of the C$^*$-algebra $K(H)$, we can get back to Corollary \ref{c for spaces with property of semi-exposition for faces} and subsequent comments whose consequences were already observed in \cite{FerGarPeVill17}.

\begin{proposition}\label{p minimal in arbitrary dimension}\cite[Proposition 2.6]{FerGarPeVill17} Let $\Delta:S(C_1(H))\to S(C_1(H'))$ be a surjective isometry, where $H$ and $H'$ are complex Hilbert spaces. Then the following statements hold:\begin{enumerate}[$(a)$]\item A subset $\mathcal{F}\subset S(C_1(H))$ is a proper norm-closed face of $\mathcal{B}_{C_1(H)}$ if and only if $\Delta(\mathcal{F})$ is.
\item $\Delta$ maps $\partial_e(\mathcal{B}_{C_1(H)})$ into $\partial_e(\mathcal{B}_{C_1(H')})$;
\item dim$(H)=$dim$(H')$.
\item For each $e_0\in \partial_e(\mathcal{B}_{C_1(H)})$ we have $\Delta(i e_0 ) = i \Delta(e_0)$ or $\Delta(i e_0 ) = -i \Delta(e_0)$;
\item For each $e_0\in \partial_e(\mathcal{B}_{C_1(H)})$ if $\Delta(i e_0 ) = i \Delta(e_0)$ {\rm(}respectively, $\Delta(i e_0 ) = -i \Delta(e_0)${\rm)} then  $\Delta(\lambda e_0 ) = \lambda \Delta(e_0)$ {\rm(}respectively, $\Delta(\lambda e_0 ) = \overline{\lambda} \Delta(e_0)${\rm)} for every $\lambda\in \mathbb{C}$ with $|\lambda|=1$.
\end{enumerate}
\end{proposition}

The strategy to solve Tingley's problem on $C_1(H)$ is based on techniques of linear algebra and geometry to obtain first a solution in the case of finite dimensional spaces.

\begin{theorem}\label{t Tingley trace class finite dim}\cite[Theorem 3.7]{FerGarPeVill17} Let $\Delta: S(C_1(H))\to S(C_1(H))$ be a surjective isometry, where $H$ is a finite dimensional complex Hilbert space. Then there exists a surjective complex linear or conjugate linear isometry $T :C_1(H)\to C_1(H)$  satisfying $\Delta(x) = T(x)$ for every $x\in S(C_1(H))$. More concretely, there exist unitary elements $u,v\in M_n(\mathbb{C}) = B(H)$ such that one of the following statements holds:\begin{enumerate}[$(a)$] \item $\Delta(x) = u x v$, for every $x\in S(C_1(H))$;
\item $\Delta(x) = u x^t v$, for every $x\in S(C_1(H))$;
\item $\Delta(x) = u \overline{x} v$, for every $x\in S(C_1(H))$;
\item $\Delta(x) = u x^* v$, for every $x\in S(C_1(H))$,
\end{enumerate} where $\overline{(x_{ij})} = (\overline{x_{ij}})$.
\end{theorem}

Surprisingly, the solution in the finite dimensional case is applied, in a very technical argument, to derive a solution to Tingley's problem for surjective isometries between the unit spheres of two spaces of trace class operators.

\begin{theorem}\label{t Tingley for trace class infinite dimension}\cite[Theorem 4.1]{FerGarPeVill17}  Let $\Delta: S(C_1(H))\to S(C_1(H))$ be a surjective isometry, where $H$ is an arbitrary complex Hilbert space. Then there exists a surjective complex linear or conjugate linear isometry $T :C_1(H)\to C_1(H)$  satisfying $\Delta (x) = T(x),$ for every $x\in S(C_1(H))$.
\end{theorem}

\subsection{Tingley's problem on von Neumann preduals}\ \smallskip

According to what is commented at the introduction, a very recent contribution by M. Mori has changed the original plans and the structure of this survey. The preprint \cite{Mori2017} contains, among other interesting results, a complete positive solution to Tingley's problem for surjective isometries between the unit spheres of von Neumann algebra preduals.

\begin{theorem}\label{t Tingleys problem for preduals of vN}\cite[Theorem 4.3]{Mori2017} Let $M$ and $N$ be von Neumann algebras, and let $\Delta : S(M_*)\to S(N_*)$ be a surjective isometry. Then there exists a (unique) surjective real linear isometry $T : M_* \to N_*$ satisfying $T(x) = \Delta(x),$ for every $x\in S(M_*)$.
\end{theorem}

It is perhaps interesting to take a brief look at the method applied by M. Mori to synthesize the surjective real linear isometry $T$. Let $\Delta : S(M_*)\to S(N_*)$ be a surjective isometry, where $M$ and $N$ are von Neumann algebras. When Corollary \ref{c for spaces with property of semi-exposition for faces} and the subsequent comments is combined with the Akemann-Pedersen theorem (see Theorem \ref{t faces AkPed}), we can conclude that for each maximal partial isometry $u\in \partial_e(\mathcal{B}_M)$ there exists a unique maximal partial isometry $T_1(u)\in \partial_e(\mathcal{B}_N)$ satisfying $\Delta(\{u\}_{\prime}) = \{T_1(u)\}_{\prime}$. This gives a bijection $T_1 : \partial_e(\mathcal{B}_M)\to \partial_e(\mathcal{B}_N)$.\smallskip

Let $(E,d)$ be a metric space. The \emph{Hausdorff distance} between two sets $\mathcal{S}_1,\mathcal{S}_2\subseteq E$ is defined by $$d_{H} (\mathcal{S}_1,\mathcal{S}_2) := \max \{ \sup_{x\in \mathcal{S}_1} \inf_{y\in \mathcal{S}_2} d(x,y), \sup_{y\in \mathcal{S}_2} \inf_{x\in \mathcal{S}_1} d(x,y)\}.$$ The lattice of partial isometries can be equipped with a distance defined by $$\delta_{H} (v,w) := d_{H} (\{v\}_{\prime}, \{w\}_{\prime}).$$

It is shown by M. Mori that this distance enjoys the following properties:

\begin{proposition}\label{p lemmas Hausdorff distance}\cite[Lemmas 4.1 and 4.2]{Mori2017} Let $M$ be a von Neumann algebra. Then the following statements hold:\begin{enumerate}[$(a)$]\item $\delta_{H} (u,v) = \|u-v\|$, for every $u\in \mathcal{U}(M)$ and every $v\in \partial_e(\mathcal{B}_{M})$;
\item An element $u\in \partial_e(\mathcal{B}_{M})$ is a unitary if and only if the set $$\widehat{M}_u := \{ e \in \partial_e(\mathcal{B}_{M}) : \delta_{H} (u, \pm e) \leq \sqrt{2}\}$$ has an isolated point with respect to the metric $\delta_H$.
\end{enumerate}
\end{proposition}

Applying Proposition \ref{p Mori faces}$(a)$ and Proposition \ref{p lemmas Hausdorff distance}$(b)$, M. Mori concludes that $T_1 (\mathcal{U} (M)) = \mathcal{U} (N),$ and by Proposition \ref{p lemmas Hausdorff distance}$(a)$, $T_1|_{\mathcal{U} (M)} : \mathcal{U} (M) \to \mathcal{U} (N)$ is a surjective isometry. The mapping $T_1$ fulfills the hypothesis of the Hatori-Moln{\'a}r theorem (see Theorem \ref{t Hatori Molnar}), and thus there exists a surjective real linear (weak$^*$-continuous) isometry $\widetilde{T}_1: M\to N$ whose restriction to $\mathcal{U} (M)$ is $T_1$. The technical arguments developed by M. Mori in the proof of \cite[Theorem 4.3]{Mori2017} finally show that the mapping $T_2 : N^*\to M^*$ defined by $$T_2(\varphi) (x) := \Re\hbox{e}\varphi (\widetilde{T}_1 (x)) - i \Re\hbox{e}\varphi (\widetilde{T}_1 (i x)), \ \varphi \in N^*, x\in M,$$ is a real linear isometry whose restriction to $N_*$ gives a surjective real linear isometry $T_2|_{N_*} : N_*\to M_*$ and $(T_2|_{M_*})^{-1} (\phi) = \Delta (\phi)$ for all $\phi$ in $M_*$.

\section{Isometries between the spheres of hermitian operators}\label{sec: Tingleys hermitian}

A second and interesting variant of Problem \ref{problem general} is obtained when $X$ and $Y$ are von Neumann algebras or C$^*$-algebras and $\mathcal{S}_1$ and $\mathcal{S}_2$ are the unit spheres of their respective hermitian parts. In this section we consider two von Neumann algebras $M$, $N$ and a surjective isometry $\Delta: S(M_{sa})\to S(N_{sa})$. Our goal will consist in showing that the same tools in \cite{FerPe17d} can be, almost literarily, applied to find a surjective complex linear isometry $T: M \to N$ satisfying $T(a^*) = T(a)^*$ for all $a\in M$ and $T(x) = \Delta (x)$ for all $x\in S(M_{sa})$.\smallskip

Given a C$^*$-algebra $A$, its hermitian part $A_{sa}$ is not, in general, a C$^*$-subalgebra of $A$.
However, $A_{sa}$ is a real closed subspace of $A$ which satisfies the hypotheses of Corollary \ref{c for spaces with property of semi-exposition for faces} (see the comments after this corollary). After applying this corollary, we find the necessity of describing the facial structure of $\mathcal{B}_{A_{sa}}$. Fortunately for us, the Akemann-Pedersen theorem (Theorem \ref{t faces AkPed}) has a forerunner in \cite[Corollary 5.1]{EdRutt86} where C.M. Edwards and G.T. R\"{u}ttimann described the facial structure of the closed unit ball of the hermitian part of every C$^*$-algebra. We recall that partial isometries in $A_{sa}$ are all elements of the form $e= p-q$, where $p$ and $q$ are orthogonal projections in $A$.

\begin{theorem}\label{t faces EdRutt}\cite[Corollary 5.1]{EdRutt86} Let $A$ be a C$^*$-algebra. Then for each norm-closed face $F$ of $\mathcal{B}_{A_{sa}}$, there exists a unique pair of orthogonal compact projections $p, q$ in $A^{**}$ such that $$F = \{x \in \mathcal{B}_{A_{sa}} : x (p - q) = p + q\} = \{p-q\}_{_{''}}$$ $$= \{ x \in \mathcal{B}_{A_{sa}} : x = (p - q) + (1- p - q) x (1-p-q) \}.$$
\end{theorem}

Combining this theorem of Edwards and R\"{u}ttimann with the above Corollary \ref{c for spaces with property of semi-exposition for faces} we easily get the following version of Theorem \ref{t first correspondence between faces and compact partial isometries in the bidual for a surjective isometry}.

\begin{theorem}\label{t first correspondence between faces and compact partial isometries in the bidual for a surjective isometry hermitian} Let $\Delta: S(A_{sa}) \to S(B_{sa})$ be a surjective isometry, where $A$ and $B$ are C$^*$-algebras. Then the following statements hold:
\begin{enumerate}[$(a)$]\item For each non-zero compact partial isometry $e\in A_{sa}^{**}$ there exists a unique (non-zero) compact partial isometry $\phi_{\Delta}^s (e)\in B_{sa}^{**}$ such that $\Delta (F_e) = F_{\phi_\Delta^s (e)},$ where $F_{e} =\left(e + (1 - e^2) \mathcal{B}_{A_{sa}^{**}} (1 - e^2)\right)\cap \mathcal{B}_{A_{sa}}$;
\item The mapping $e\mapsto \phi_\Delta^s (e)$ defines an order preserving bijection between the sets of non-zero compact partial isometries in $A_{sa}^{**}$ and the set of non-zero compact partial isometries in $B_{sa}^{**}$;
\item $\phi_{\Delta}^s$ maps minimal partial isometries in $A_{sa}^{**}$ to minimal partial isometries in $B_{sa}^{**}$.
\end{enumerate}
\end{theorem}

The arguments in the proofs of \cite[Theorems 2.7, 2.8 and 2.11 and Proposition 2.12]{FerPe17d} literarily works to obtain the following four results.

\begin{theorem}\label{t surjective isometries map partial isometries into points of strong subdiff hermitian}\cite[Theorem 2.7]{FerPe17d}
Let $\Delta: S(A_{sa}) \to S(B_{sa})$ be a surjective isometry, where $A$ and $B$ are C$^*$-algebras. Let $e$ be a non-zero partial isometry in $A_{sa}$. Then $1$ is isolated in the spectrum of $|\Delta(e)|$.
\end{theorem}

\begin{theorem}\label{t A hermitian}\cite[Theorem 2.8]{FerPe17d}
Let $\Delta: S(A_{sa}) \to S(B_{sa})$ be a surjective isometry, where $A$ and $B$ are C$^*$-algebras. Then $\Delta$ maps non-zero partial isometries in $A_{sa}$ into non-zero partial isometries in $B_{sa}$. Moreover, for each non-zero partial isometry $e$ in $A_{sa}$, we have $\phi_{\Delta}^s (e) = \Delta(e)$, where $\phi_{\Delta}^s$ is the mapping given by Theorem \ref{t first correspondence between faces and compact partial isometries in the bidual for a surjective isometry hermitian}, and  there exists a surjective (real) linear isometry $$T_{e} : (1 - e^2) A_{sa}  (1 - e^2)\to  {(1 - \Delta(e)^2) B_{sa} (1 - \Delta(e)^2)}$$ such that $$ \Delta(e + x) = \Delta(e) + T_e (x), \hbox{ for all $x$ in } \mathcal{B}_{(1 - e^2) A_{sa}  (1 - e^2)}.$$ In particular the restriction of $\Delta$ to the face $F_{e} = e + (1 - e^2) \mathcal{B}_{A_{sa}} (1 - e^2)$ is a real affine function.
\end{theorem}

\begin{theorem}\label{t C hermitian}\cite[Theorem 2.11]{FerPe17d} Let $\Delta: S(A_{sa}) \to S(B_{sa})$ be a surjective isometry, where $A$ and $B$ are C$^*$-algebras. Then, for each non-zero compact partial isometry $e$ in $A_{sa}^{**}$ we have $\phi_\Delta^s (-e) = -\phi_\Delta^s (e)$, where $\phi_\Delta^s$ is the mapping given by Theorem \ref{t first correspondence between faces and compact partial isometries in the bidual for a surjective isometry hermitian}. Consequently, for each non-zero partial isometry $e\in A_{sa}$ we have $\Delta(-e) = -\Delta(e)$.
\end{theorem}

\begin{proposition}\label{p algebraic elements hermitian}\cite[Proposition 2.12]{FerPe17d} Let $\Delta: S(A_{sa}) \to S(B_{sa})$ be a surjective isometry, where $A$ and $B$ are C$^*$-algebras. Then the following statements hold:\begin{enumerate}[$(a)$]\item For each non-zero partial isometry $v$ in $A_{sa}$, the surjective real linear isometry $$T_v : (1 - v^2) A_{sa}  (1 - v^2)\to  {(1 - \Delta(v)^2) B_{sa} (1 - \Delta(v)^2)}$$ given by Theorem \ref{t A hermitian} satisfies $\Delta (e) = T_v(e),$ for every non-zero partial isometry $e\in (1 - v^2) A_{sa}  (1 - v^2)$;
\item Let $w_1,\ldots,w_n$ be mutually orthogonal non-zero partial isometries in $A_{sa}$, and let $\lambda_1,\ldots,\lambda_n$ be real numbers with $1=|\lambda_1|\geq\max\{|\lambda_j|\}$. Then $$\Delta\left(\sum_{j=1}^n \lambda_j w_j\right) = \sum_{j=1}^n \lambda_j \Delta\left(w_j\right);$$
\item Suppose $v,w$ are mutually orthogonal non-zero partial isometries in $A_{sa}$ then $T_{v} (x) = T_{w} (x)$ for every $x\in \{v\}^{\perp} \cap \{w\}^{\perp}$;
\item If $A$ is a von Neumann algebra, then for each non-zero partial isometry $v$ in $A_{sa}$ we have $\Delta(x) = T_v (x)$ for every $x\in S({(1 - vv^*) A_{sa}  (1 - v^* v)}).$
\end{enumerate}
\end{proposition}

Back to our goal, we observe that the case of $M_2(\mathbb{C})$ of all $2\times 2$ matrices with complex entries must be treated independently.

\begin{proposition}\label{p Tingley hermitian M2} Let $A= M_2 (\mathbb{C})$, $B$ a C$^*$-algebra, and let $\Delta: S(A_{sa}) \to S(B_{sa})$ be a surjective isometry. Then there exists a surjective complex linear isometry $T:A\to B$ satisfying $T(a^*) = T(a)^*,$ for all $a\in A,$ and $T(x) = \Delta (x),$ for all $x\in S(A_{sa})$.
\end{proposition}

\begin{proof} Since $A$ is finite dimensional, it follows from the hypotheses that $S(B_{sa})$ is compact, and hence $B$ is finite dimensional. Having in mind that the rank of a von Neumann algebra $M$ is the cardinal of a maximal set of mutually orthogonal projections, Proposition \ref{p algebraic elements hermitian} assures that $B$ must have rank 2. Therefore $B= \mathbb{C}\oplus^{\infty}\mathbb{C}$ or $B= M_2(\mathbb{C})$. We shall show that the first case is impossible.\smallskip

Suppose $B= \mathbb{C}\oplus^{\infty}\mathbb{C}$. We pick two orthogonal minimal projections $p_1= \left(
                                                                                                    \begin{array}{cc}
                                                                                                      1 & 0 \\
                                                                                                      0 & 0 \\
                                                                                                    \end{array}
                                                                                                  \right)
$ and $p_2 = \left(
               \begin{array}{cc}
                 0 & 0 \\
                 0 & 1 \\
               \end{array}
             \right),$ and a symmetry $s = \left(
                                             \begin{array}{cc}
                                               0 & 1 \\
                                               1 & 0 \\
                                             \end{array}
                                           \right)$
              in $A$.\smallskip

By Theorem \ref{t first correspondence between faces and compact partial isometries in the bidual for a surjective isometry hermitian}$(b)$ and Proposition \ref{p algebraic elements hermitian}, $\Delta(p_1)$ and $\Delta(p_2)$ are orthogonal minimal partial isometries in $B_{sa},$ and $\Delta(s)$ is a symmetry in $B$. We can assume, without loss of generality, that $\Delta(p_1) = (\pm 1,0)$,  $\Delta(p_2) = (0,\pm 1)$, and $\Delta(s) = (\sigma_1 ,\sigma_2 ),$ where $\sigma_1,\sigma_2\in\{\pm 1\}$. By hypotheses, $$ \frac{1+\sqrt{5}}{2}= \left\|\left(
                                                                             \begin{array}{cc}
                                                                               1 & -1 \\
                                                                               -1 & 0 \\
                                                                             \end{array}
                                                                           \right)\right\|= \| p_1-s\|= \| \Delta(p_1) - \Delta(s) \| $$ $$= \|(\pm 1,0)- (\sigma_1 ,\sigma_2 ) \| \in \{1,2\},$$ which is impossible. Therefore, $B= M_2(\mathbb{C})$.\smallskip

Let us take a surjective complex linear and symmetric isometry $T_1: M_2(\mathbb{C}) \to M_2(\mathbb{C})$ mapping $\Delta(p_1)$ and $\Delta(p_2)$ to $p_1$ and $p_2$, respectively. We set $\Delta_1 = T_1\circ \Delta$. Then $\Delta_1 : S(A_{sa})\to S(B_{sa})$ is a surjective isometry with $\Delta_1 (p_i) = p_i$ for $i=1,2$.\smallskip

An arbitrary pair of orthogonal minimal projections in $A_{sa}$ writes in the form $q_1= \left(
\begin{array}{cc}
s_0 & \!\!\lambda \sqrt{s_0(1-s_0)} \\
\!\!\overline{\lambda} \sqrt{s_0(1-s_0)} &\! 1-s_0 \\
\end{array}
\right)
$ and $q_2 \!=\! \left(
\begin{array}{cc}
1- s_0 & \!\!\!-\lambda \sqrt{s_0(1-s_0)} \\
\!\!\!-\overline{\lambda} \sqrt{s_0(1-s_0)} &\! \! s_0 \\
\end{array}
\right)$ for a unique $s_0\in(0,1)$ and a unique $\lambda\in \mathbb{T}$ (the cases $s_0=0,1$ give $p_1$ and $p_2$). By Theorem \ref{t first correspondence between faces and compact partial isometries in the bidual for a surjective isometry hermitian}$(b)$ and Proposition \ref{p algebraic elements hermitian}, $\Delta_1(q_1)$ and $\Delta_1(q_2)$ are orthogonal minimal partial isometries in $B_{sa}$.  It is well known that $\Delta_1 (q_1)\! =\! \pm \left(
\begin{array}{cc}
\!\!\! t_0 & \!\!\!\! \mu \sqrt{t_0(1-t_0)} \\
\!\!\!\overline{\mu} \sqrt{t_0(1-t_0)} & \!\!\! 1-t_0 \\                  
\end{array}                                                                                                                  \right)$
for a unique $t_0\in [0,1]$ and a unique $\mu\in \mathbb{T}$ (compare \cite[Theorem 1.3]{RaSin} or \cite[\S 3]{Ped68}).\smallskip

If $\Delta_1 (q_1) = -\left(
\begin{array}{cc}
t_0 & \mu \sqrt{t_0(1-t_0)} \\
\overline{\mu}\sqrt{t_0(1-t_0)} & 1-t_0 \\
\end{array}
\right)$, then by hypothesis, $$1+\sqrt{t_0}= \left\| \left(
\begin{array}{cc}
t_0+1 & \mu \sqrt{t_0(1-t_0)} \\
\overline{\mu} \sqrt{t_0(1-t_0)} & 1-t_0 \\
\end{array}
\right)\right\|=\| - \Delta_1 (q_1) +\Delta_1 (p_1)\|$$ $$=\|-q_1+p_1\|=\|q_1-p_1\|=\left\| \left(
\begin{array}{cc}
s_0-1 & \lambda \sqrt{s_0(1-s_0)} \\
\overline{\lambda}\sqrt{s_0(1-s_0)} & 1-s_0 \\
\end{array}
\right)\right\| =\sqrt{(1-s_0)},$$
which is impossible.\smallskip

If $\Delta_1 (q_1) = \left(
\begin{array}{cc}
\!\!\! t_0 & \!\!\!\! \mu \sqrt{t_0(1-t_0)} \\
\!\!\!\overline{\mu} \sqrt{t_0(1-t_0)} & \!\!\! 1-t_0 \\
\end{array}                                                                                                                  \right)$, then by hypothesis, $$\sqrt{(1-t_0)}= \left\| \left(
\begin{array}{cc}
t_0-1 & \mu \sqrt{t_0(1-t_0)} \\
\overline{\mu} \sqrt{t_0(1-t_0)} & 1-t_0 \\                                         
 \end{array}
\right)\right\|=\| \Delta_1 (q_1) -\Delta_1 (p_1)\|$$ $$=\|q_1-p_1\|=\left\| \left(              \begin{array}{cc}
s_0-1 & {\lambda} \sqrt{s_0(1-s_0)} \\
\overline{\lambda} \sqrt{s_0(1-s_0)} & 1-s_0 \\                    
\end{array}
\right)\right\| =\sqrt{(1-s_0)},$$
which implies that $t_0 = s_0$. That is, for each $s_0\in [0,1]$ and $\lambda\in \mathbb{T}$, there exists a unique $\mu\in \mathbb{T}$ such that  \begin{equation}\label{eq first part} \Delta_1 \left( \left(
\begin{array}{cc}
s_0 & \!\!\lambda \sqrt{s_0(1-s_0)} \\
\!\!\overline{\lambda} \sqrt{s_0(1-s_0)} &\! 1-s_0 \\
\end{array}
\right)\right) = \left(
\begin{array}{cc}
s_0 & \!\!\mu \sqrt{s_0(1-s_0)} \\
\!\!\overline{\mu} \sqrt{s_0(1-s_0)} &\! 1-s_0 \\
\end{array}
\right)
\end{equation} In particular, $\Delta_2 \left(\left(
                                                             \begin{array}{cc}
                                                               \frac12 & \frac12 \\
                                                               \frac12 & \frac12 \\
                                                             \end{array}
                                                           \right)
\right) =  \left(
                                                             \begin{array}{cc}
                                                               \frac12 & \mu_0\frac12 \\
                                                               \overline{\mu_0}\frac12 & \frac12 \\
                                                             \end{array}
                                                           \right),$ for certain $\mu_0\in \mathbb{T}$.\smallskip

Let us take a surjective complex linear symmetric isometry $T_2: M_2(\mathbb{C}) \to M_2(\mathbb{C})$ satisfying $T_2(p_j) = p_j$ for every $j=1,2$ and $T_2 \Delta_2 \left(\left(
                                                             \begin{array}{cc}
                                                               \frac12 & \frac12 \\
                                                               \frac12 & \frac12 \\
                                                             \end{array}
                                                           \right)
\right) = \left(
                                                             \begin{array}{cc}
                                                               \frac12 & \frac12 \\
                                                               \frac12 & \frac12 \\
                                                             \end{array}
                                                           \right).$ We set $\Delta_2 = T_2 \circ \Delta_1 : S(A_{sa})\to S(B_{sa}).$
Proposition \ref{p algebraic elements hermitian}$(b)$ applied to $\Delta_2$ gives $$1=\Delta_2(p_1) + \Delta_2(p_2) = \Delta_2 (1) = \Delta_2\left(\left(
                                                             \begin{array}{cc}
                                                               \frac12 & \frac12 \\
                                                               \frac12 & \frac12 \\
                                                             \end{array}
                                                           \right)\right) + \Delta_2\left(\left(
                                                             \begin{array}{cc}
                                                               \frac12 & -\frac12 \\
                                                               -\frac12 & \frac12 \\
                                                             \end{array}
                                                           \right)\right),$$ which assures that $\Delta_2\left(\left(
                                                             \begin{array}{cc}
                                                               \frac12 & -\frac12 \\
                                                               -\frac12 & \frac12 \\
                                                             \end{array}
                                                           \right)\right)=\left(
                                                             \begin{array}{cc}
                                                               \frac12 & -\frac12 \\
                                                               -\frac12 & \frac12 \\
                                                             \end{array}
                                                           \right).$ Let us denote $r_1 = \left(
                                                             \begin{array}{cc}
                                                               \frac12 & \frac12 \\
                                                               \frac12 & \frac12 \\
                                                             \end{array}
                                                           \right),$ and $r_2 = 1-r_1$. 
A new application of Proposition \ref{p algebraic elements hermitian}$(b)$ gives $$\Delta_2 ( r_1 - r_2 ) = \Delta_2 ( r_1 ) - \Delta_2 (  r_2 ) = r_1 -r_2.$$ Take an arbitrary projection $q_1 = \left(
\begin{array}{cc}
s_0 & \!\!\lambda \sqrt{s_0(1-s_0)} \\
\!\!\overline{\lambda} \sqrt{s_0(1-s_0)} &\! 1-s_0 \\
\end{array}
\right)$ with $s_0\in (0,1)$ and $\lambda\in \mathbb{T}$. We deduce from the hypothesis and \eqref{eq first part} (applied to $\Delta_2$) that $$ \frac{1+\sqrt{5- 8 \Re\hbox{e}(\lambda) \sqrt{s_0(1-s_0)} }}{2}= \left\| \left(
\begin{array}{cc}
s_0 & \!\!\lambda \sqrt{s_0(1-s_0)} \\
\!\!\overline{\lambda} \sqrt{s_0(1-s_0)} &\! 1-s_0 \\
\end{array}
\right) - \left(
                                                             \begin{array}{cc}
                                                               0 & 1 \\
                                                               1 & 0 \\
                                                             \end{array}
                                                           \right) \right\| $$ $$=  \left\| \Delta_2 \left(
\begin{array}{cc}
s_0 & \!\!\lambda \sqrt{s_0(1-s_0)} \\
\!\!\overline{\lambda} \sqrt{s_0(1-s_0)} &\! 1-s_0 \\
\end{array}
\right) - \Delta_2 \left(
                                                             \begin{array}{cc}
                                                               0 & 1 \\
                                                               1 & 0 \\
                                                             \end{array}
                                                           \right) \right\| $$ $$= \left\| \left(
\begin{array}{cc}
s_0 & \!\!\mu \sqrt{s_0(1-s_0)} \\
\!\!\overline{\mu} \sqrt{s_0(1-s_0)} &\! 1-s_0 \\
\end{array}
\right) - \left(
                                                             \begin{array}{cc}
                                                               0 & 1 \\
                                                               1 & 0 \\
                                                             \end{array}
                                                           \right) \right\|=  \frac{1+\sqrt{5- 8 \Re\hbox{e}(\mu) \sqrt{s_0(1-s_0)} }}{2},$$
which assures that the scalar $\mu$ in \eqref{eq first part} for $\Delta_2$ must satisfy $\mu = \lambda$ or $\mu = \overline{\lambda}$. Consequently, for each $s_0\in (0,1)$ and $\lambda\in \mathbb{T}$ we have  \begin{equation}\label{eq first part b} \Delta_2 \left( \left(
\begin{array}{cc}
s_0 & \!\!\lambda \sqrt{s_0(1-s_0)} \\
\!\!\overline{\lambda} \sqrt{s_0(1-s_0)} &\! 1-s_0 \\
\end{array}
\right)\right) = \left(
\begin{array}{cc}
s_0 & \!\!\lambda \sqrt{s_0(1-s_0)} \\
\!\!\overline{\lambda} \sqrt{s_0(1-s_0)} &\! 1-s_0 \\
\end{array}
\right)
\end{equation} or $$\Delta_2 \left( \left(
\begin{array}{cc}
s_0 & \!\!\lambda \sqrt{s_0(1-s_0)} \\
\!\!\overline{\lambda} \sqrt{s_0(1-s_0)} &\! 1-s_0 \\
\end{array}
\right)\right) = \left(
\begin{array}{cc}
s_0 & \!\!\overline{\lambda} \sqrt{s_0(1-s_0)} \\
\!\! {\lambda} \sqrt{s_0(1-s_0)} &\! 1-s_0 \\
\end{array}
\right). $$\smallskip

We can also deduced above and Proposition \ref{p algebraic elements hermitian}$(b)$  that $$\Delta_2 \left( \left(
                                                    \begin{array}{cc}
                                                      0 & i \\
                                                      -i & 0 \\
                                                    \end{array}
                                                  \right)
\right) = \left(
                                                    \begin{array}{cc}
                                                      0 & i \\
                                                      -i & 0 \\
                                                    \end{array}
                                                  \right), \ \hbox{ or } \Delta_2 \left( \left(
                                                    \begin{array}{cc}
                                                      0 & i \\
                                                      -i & 0 \\
                                                    \end{array}
                                                  \right)
\right) = \left(
                                                    \begin{array}{cc}
                                                      0 & -i \\
                                                      i & 0 \\
                                                    \end{array}
                                                  \right).$$
                                                  
Suppose first that $\Delta_2 \left( \left(
                                                    \begin{array}{cc}
                                                      0 & i \\
                                                      -i & 0 \\
                                                    \end{array}
                                                  \right)
\right) = \left(
                                                    \begin{array}{cc}
                                                      0 & i \\
                                                      -i & 0 \\
                                                    \end{array}
                                                  \right).$ Given $s_0\in (0,1)$ and $\lambda\in \mathbb{T}$, we have $$ \frac{1+\sqrt{5+ 8 \Im\hbox{m}(\lambda) \sqrt{s_0(1-s_0)} }}{2}= \left\| \left(
\begin{array}{cc}
s_0 & \!\!\lambda \sqrt{s_0(1-s_0)} \\
\!\!\overline{\lambda} \sqrt{s_0(1-s_0)} &\! 1-s_0 \\
\end{array}
\right) - \left(
                                                             \begin{array}{cc}
                                                               0 & i \\
                                                               -i & 0 \\
                                                             \end{array}
                                                           \right) \right\|,$$ 
$$ \frac{1+\sqrt{5- 8 \Im\hbox{m}(\lambda) \sqrt{s_0(1-s_0)} }}{2}= \left\| \left(
\begin{array}{cc}
s_0 & \!\!\overline{\lambda} \sqrt{s_0(1-s_0)} \\
\!\! {\lambda} \sqrt{s_0(1-s_0)} &\! 1-s_0 \\
\end{array}
\right) - \left(
                                                             \begin{array}{cc}
                                                               0 & i \\
                                                               -i & 0 \\
                                                             \end{array}
                                                           \right) \right\|,$$
and thus, \eqref{eq first part b} and the hypothesis prove that $$\Delta_2(q_1)= \Delta_2 \left( \left(
\begin{array}{cc}
s_0 & \!\!\lambda \sqrt{s_0(1-s_0)} \\
\!\!\overline{\lambda} \sqrt{s_0(1-s_0)} &\! 1-s_0 \\
\end{array}
\right)\right) $$ $$= \left(
\begin{array}{cc}
s_0 & \!\!\lambda \sqrt{s_0(1-s_0)} \\
\!\!\overline{\lambda} \sqrt{s_0(1-s_0)} &\! 1-s_0 \\
\end{array}
\right) = q_1,$$ for every $q_1$ as above. Let $q_2 = 1-q_1$. By Proposition \ref{p algebraic elements hermitian}$(b)$ we also have  $$1=\Delta_2(p_1) + \Delta_2(p_2) = \Delta_2 (1) = \Delta_2(q_1) + \Delta_2(q_2),$$ which assures that $\Delta_2(q_2)=q_2.$ We have therefore proved that $\Delta_2(q_i)=q_i,$ for every pair of orthogonal minimal projections $q_1,q_2$ in $A_{sa}$. Since every element $x$ in $S(A_{sa})$ can be written as a linear combination of the form $\displaystyle x= \sum_{j=1}^2 \mu_j q_j$, where $q_1$ and $q_2$ are orthogonal minimal projections in $A_{sa}$, $\mu_j\in\mathbb{R}$ and $\max\{|\mu_j|\}=1$, a new application of Proposition \ref{p algebraic elements hermitian}$(b)$ gives $$\Delta_2(x)= \Delta_2 \left(\sum_{j=1}^2 \mu_j q_j \right)= \sum_{j=1}^2 \mu_j \Delta_2(q_j) = \sum_{j=1}^2 \mu_j q_j =x.$$ This shows that $\Delta_2 (x) = x,$ for every $x$ in $S(A_{sa})$, and hence $\Delta (x) =T_2^{-1} T_1^{-1} (x)$, for every $x$ in $S(A_{sa})$.\smallskip

Assume now that $\Delta_2 \left( \left(
                                                    \begin{array}{cc}
                                                      0 & i \\
                                                      -i & 0 \\
                                                    \end{array}
                                                  \right)
\right) = \left(
                                                    \begin{array}{cc}
                                                      0 & -i \\
                                                      i & 0 \\
                                                    \end{array}
                                                  \right).$ 
Similar arguments to those given above show that, in this case, we have $$\Delta_2(q_1)= \Delta_2 \left( \left(
\begin{array}{cc}
s_0 & \!\!\lambda \sqrt{s_0(1-s_0)} \\
\!\!\overline{\lambda} \sqrt{s_0(1-s_0)} &\! 1-s_0 \\
\end{array}
\right)\right) $$ $$= \left(
\begin{array}{cc}
s_0 & \!\!\overline{\lambda} \sqrt{s_0(1-s_0)} \\
\!\! {\lambda} \sqrt{s_0(1-s_0)} &\! 1-s_0 \\
\end{array}
\right) = \overline{q_1},$$ for every minimal projection $q_1$ as above, where $\overline{(x_{ij})}=(\overline{x_{i,j}})$, and $\Delta_2 (x) = \overline{x},$ for every $x$ in $S(A_{sa})$. Therefore, $\Delta (x) =T_2^{-1} T_1^{-1} (\overline{x})$, for every $x$ in $S(A_{sa})$. Denoting $S= T_2^{-1} T_1^{-1}$ we have a complex linear and symmetric isometry $S: M_2(\mathbb{C}) \to M_2(\mathbb{C})$. We define $T: M_2(\mathbb{C}) \to M_2(\mathbb{C})$ by $T(h + i k) := S(\overline{h}) + i S(\overline{k}) = S(\overline{h} + i \overline{k}) = S(\overline{h- i k}) = S(\overline{(h+ i k)^*}) = S((h+ik)^t),$ for every $h,k\in A_{sa}$, which provides the mapping $T$ in the statement of the proposition.
\end{proof}

We can state now the desired result and its proof, where we show that the synthesis of a surjective isometry is even easier in this setting.

\begin{theorem}\label{t Tingley hermitian von Neumann} Let $\Delta : S(M_{sa})\to S(N_{sa})$ be a surjective isometry, where $M$ and $N$ are von Neumann algebras. Then there exists a surjective complex linear isometry $T: M\to N$ satisfying $T(a^*) = T(a)^*,$ for all $a\in M,$ and $T(x) = \Delta (x),$ for all $x\in S(M_{sa})$.
\end{theorem}

\begin{proof} We shall distinguish the following three cases,
\begin{enumerate}[$(1)$]\item $M$ contains no type $I_2$ von Neumann factors;
\item $M$ contains a type $I_2$ von Neumann factor but $M$ is not a type $I_2$ von Neumann factor;
\item $M$ is a type $I_2$ von Neumann factor.
\end{enumerate}

Case $(3)$ is solve by Proposition \ref{p Tingley hermitian M2}.\smallskip

Case $(2)$. We can assume that $M= J_1\oplus J_2$, where $J_1$ and $J_2$ are non-zero orthogonal weak$^*$ closed ideals of $M$ and $J_1 = M_2(\mathbb{C})$. We can now mimic the arguments we gave in the solution to Tingley's problem for compact operators in page \pageref{eq proof compact with several subfactors}. Let us take two non-zero projections $p_1$ in $J_1$ and $p_2\in J_2$, and define a mapping $T : M  \to N$ given by $$T(x) = T_{p_1} (\pi_2 (x)) + T_{p_{2}} (\pi_1 (x))$$ where $\pi_1$ and $\pi_2$ stand for the canonical projections of $M$ onto $J_1$ and $J_2$, respectively, and $T_{p_1}$ and $T_{p_2}$ are the surjective weak$^*$ continuous complex linear and symmetric isometries given by Theorem \ref{t A hermitian}. The mapping $T$ is complex linear and weak$^*$ continuous because $T_{p_1}$ and $T_{p_2}$ are. Any projection $p$ in $M$ can be written in the form $p = p_1 + p_2$ where $p_j$ is a projection in $J_1$. Let us pick an algebraic element $x$ in $S(M_{sa})$ which can be written in the form $\displaystyle x=\sum_{j=1} \alpha_j p_j + \sum_{k=1} \beta_k q_k$, where $p_j,q_k$ are mutually ortogonal non-zero projections in $M_{sa}$, $\alpha_j$, $\beta_k\in \mathbb{R}\backslash\{0\},$ $\max\{|\alpha_j|,|\beta_k|\}=1$, $p_j\in J_1$ and $q_k\in J_2$ for all $j$, $k.$ By definition of $T$ and Proposition \ref{p algebraic elements hermitian}$(b)$ we have $$ \Delta(x) = \sum_{j=1} \alpha_j \Delta(p_j) + \sum_{k=1} \beta_k \Delta(q_k) =\sum_{j=1} \alpha_j T_{p_2}(p_j) + \sum_{k=1} \beta_k T_{p_1} (q_k)$$ $$ = \sum_{j=1} \alpha_j T(p_j) + \sum_{k=1} \beta_k T(q_k) = T(x).$$ The norm density of this kind of algebraic elements $x$ in $S(M_{sa})$ together with the norm continuity of $T$ and $\Delta$ prove that $T(x) = \Delta(x)$ for all $x\in S(M)$.\smallskip

Case $(1)$. $M$ contains no type $I_2$ von Neumann factors. Let us define a vector measure on the lattice $\mathcal{P}roj(M)$ of all projections of $M$ defined by $\mu : \mathcal{P}roj(M) \to N$, $\mu (p) = \Delta (p) $ if $p\in S(M)$ and $\mu(0) =0$. Proposition \ref{p algebraic elements hermitian}$(b)$ assure that $\mu$ is finitely additive, that is $$\mu \left(\sum_{j=1}^m p_j\right) = \sum_{j=1}^m \mu(p_j),$$ whenever $p_1,\ldots, p_m$ are mutually orthogonal projections in $M$. We further have $\|\mu(p)\|\leq 1$ for every $p\in \mathcal{P}roj(M)$. By the Bunce-Wright-Mackey-Gleason theorem (see \cite[Theorem A]{BuWri92} or \cite[Theorem A]{BuWri94}) there exists a bounded (complex) linear operator $T: M \to N$ satisfying $T(p) = \mu(p) = \Delta (p),$ for every $p\in \mathcal{P}roj(M)\backslash\{0\}$.  By definition $T(p)\in N_{sa}$ for every projection $p$ in $M$. Therefore $T$ is a symmetric map, that is, $T(a^*) = T(a)^*$ for all $a\in M$.\smallskip

Finally, Proposition \ref{p algebraic elements hermitian}$(b)$ also guarantees that $\Delta$ and $T$ coincide on algebraic elements in $S(M_{sa})$ which can be written as finite real linear combinations of mutually orthogonal projections. Since this kind of algebraic elements are norm dense in $S(M_{sa})$, we deduce from the norm continuity of $\Delta$ and $T$ that $T(x) = \Delta(x)$ for all $x\in S(M_{sa})$.
\end{proof}

\begin{remark}\label{remark Mori hermitian}{\rm After completing the writing of this chapter, the preprint by M. Mori \cite{Mori2017} became available in arxiv. Section 5 in the just quoted paper is devoted to study Theorem \ref{t Tingley hermitian von Neumann} with a different proof based on a theorem of Dye on orthoisomorphisms (see \cite[\S 5]{Mori2017} and \cite{Dye55}). So, Theorem \ref{t Tingley hermitian von Neumann} should be also credited to M. Mori. It is surprising that the arguments developed by Mori find a similar obstacle with type $I_2$ von Neumann factors when applying Dye's theorem.  To solve the difficulties Mori build a analogue to our Proposition \ref{p Tingley hermitian M2} in \cite[Proposition 5.2 and its proof]{Mori2017}. The proof of Proposition \ref{p Tingley hermitian M2} is a bit simpler with pure geometry-linear algebra arguments.}
\end{remark}

\begin{open problem}\label{open problem Tingley hermitian} Let $\Delta : S(A_{sa})\to S(B_{sa})$ be a surjective isometry between the unit spheres of the hermitian parts of two C$^*$-algebras. Does $\Delta$ admits an extension to a surjective complex linear isometry from $A$ onto $B$?
\end{open problem}

\section{Isometries between the spheres of positive operators}\label{sec: Tingleys positive}

Contrary to the results revised in previous sections, in the third variant of Problem \ref{problem general} treated in this survey the theory on the facial structure of a C$^*$-algebra revised in section \ref{sec: geometric background} will not play any role. Let us estate the concrete statement. Given a subset $B$ of a Banach space $X$, the symbol $S(B)$ will stand for the intersection of $B$ and $S(X)$. Given a C$^*$-algebra $A$, the symbol $A^+$ will denote the cone of positive elements in $A$, while $S(A^+)$ will stand for the sphere of positive norm-one operators. The concrete variant of Problem \ref{problem general} reads as follows.

\begin{problem}\label{problem positive Tingley's problem} Let $\Delta : S(X^+)\to S(Y^+)$ be a surjective isometry, where $X$ and $Y$ are Banach spaces which can be regarded as linear subspaces two C$^*$-algebras $A$ and $B,$ $S(X^+) = S(X)\cap A^+$ and $S(Y^+) = S(Y)\cap B^+$. Does $\Delta$ admit an extension to a surjective complex linear isometry $T : X\to Y$?
\end{problem}

Problem \ref{problem positive Tingley's problem} is too general. We can easily find non isomorphic Banach spaces $X$ and $Y$ which are linear subspaces of two C$^*$-algebras $A$ and $B$, for which $S(X^+)$ and $S(Y^+)$ reduce to a single point.\smallskip

Before dealing with the historical background and forerunners, we shall make some observations. If we have a surjective isometry  $\Delta : S(A^+)\to S(B^+)$ between the spheres of positive elements in two arbitrary C$^*$-algebras the application of Theorems \ref{t faces ChengDong11} and \ref{t ChengDong for general faces} is non-viable because $A^+$ and $B^+$ are not Banach spaces.\smallskip

Another comment, the hypotheses in Problem \ref{problem positive Tingley's problem} are strictly weaker than those in Theorems \ref{t Tanaka finite vN}, \ref{thm Tingley compact Cstaralgebras}, \ref{thm Tyngley ellinfty sums}, \ref{t Tingley von Neumann}, \ref{t Tingley for trace class infinite dimension}, \ref{t Tingleys problem for preduals of vN}, and \ref{t Tingley hermitian von Neumann}. However, the conclusion is also weaker because we need to find a surjective isometry $T:A\to B$ whose restriction to $S(A^+)$ coincides with $\Delta$, we do not have to show that $T$ and $\Delta$ coincide on the whole $S(A)$ nor on $S(A_{sa})$. That is, the synthesis of the mapping $T$ is, a priori, easier at the cost of loosing the main geometric tools.\smallskip

We can now go survey the main achievements in this line. Let us recall some terminology. According to the notation in previous sections, we shall denote by $(C_p(H), \|\cdot\|_p)$ the Banach space of all $p$-Schatten-von Neumann operators on a complex Hilbert space $H$, where $1\leq p\leq \infty$. For $p=1$ we find the space of trace class operators. By an standard abuse of notation we identify $C_{\infty}(H)$ with $B(H)$. Let the symbol $C_p(H)^+$ denote the set of all positive operators in $C_p(H)$. The elements in the set $S(C_p(H)^+) = S(C_p(H))\cap C_p(H)^+$ are usually called \emph{density operators}.\smallskip

Our first result, which was obtained by L. Moln{\'a}r and W. Timmermann in \cite{MolTim2003}, provides a complete positive solution to Problem \ref{problem positive Tingley's problem} for the space $C_1(H)$ of trace class operators on an arbitrary complex Hilbert space $H$.

\begin{theorem}\label{t Molnar Timmermann sphere positive C1 arbitrary H}\cite[Theorem 4]{MolTim2003} Let $H$ be an arbitrary complex Hilbert space. Then every surjective isometry $\Delta : S(C_1(H)^+) \to S(C_1(H)^+)$ admits a unique extension to a surjective complex linear isometry on $C_1(H)$.
\end{theorem}

In 2012, G. Nagy and L. Moln{\'a}r solve Problem \ref{problem positive Tingley's problem} in the finite dimensional case for every $1\leq p$.

\begin{theorem}\label{t Nagy Molnar sphere positive Cp finite dim}\cite[Theorem 1]{MolNag2012} Let $H$ be a finite dimensional complex Hilbert space, and let $\infty >p\geq 1$. Then every isometry $\Delta : S(C_p(H)^+) \to S(C_p(H)^+)$ admits a unique extension to a surjective complex linear isometry on $C_p(H)$.
\end{theorem}

Let us observe that the mapping $\Delta$ in the above theorem is not assumed to be surjective a priori. However, as a consequence of the result $\Delta$ is surjective.\smallskip

Theorem \ref{t Nagy Molnar sphere positive Cp finite dim} was extended by G. Nagy to arbitrary complex Hilbert spaces in \cite{Nagy2013}.

\begin{theorem}\label{t Nagy Cp arbitrary}\cite[Theorem 1]{Nagy2013} Let $H$ be an arbitrary complex Hilbert space, and let $p\in (1,\infty)$. Then every isometry $\Delta : S(C_p(H)^+) \to S(C_p(H)^+)$ admits a unique extension to a surjective complex linear isometry on $C_p(H)$.
\end{theorem}

Problem \ref{problem positive Tingley's problem} has been explored, in a very recent paper due G. Nagy, for surjective isometries $\Delta : S(B(H)^+) \to S(B(H)^+)$ under the hypothesis of $H$ being finite dimensional. In the paper \cite{Nagy2017} we can find the following result.

\begin{theorem}\label{t Nagy B(H) positive finite dim}\cite[Theorem]{Nagy2017} Let $H$ be a finite dimensional complex Hilbert space, and let $\Delta : S(B(H)^+)\to S(B(H)^+)$ be an isometry. Then $\Delta$ is surjective and there exists a (unique) surjective complex linear isometry $T : B(H) \to B(H)$ satisfying $T(x) = \Delta(x),$ for all $x\in S(B(H)^+)$.
\end{theorem}

The arguments developed by Nagy in the paper \cite{Nagy2017} develop some interesting tools and results in the finite dimensional setting. Some of them have been successfully extended to arbitrary dimensions. Let $E$ and $P$ be subsets of a Banach space $X$. Following the notation employed in the recent paper \cite{Per2017}, the \emph{unit sphere around $E$ in $P$} is defined as the set $$Sph(E;P) :=\left\{ x\in P : \|x-b\|=1 \hbox{ for all } b\in E \right\}.$$ To simplify the notation, given a C$^*$-algebra $A$, and a subset $E\subset A$ we shall write $Sph^+ (E)$ or $Sph_A^+ (E)$ for the set $Sph(E;S(A^+))$.\smallskip

In \cite[Proof of Claim 1]{Nagy2017} G. Nagy proves that if $H$ is a finite dimensional complex Hilbert space, and $a$ is a positive norm-one element in $B(H)=M_n(\mathbb{C})$, then $$ \hbox{$a$ is a projection} \hbox{ if, and only if, } Sph^+_{M_n(\mathbb{C})} \left( Sph^+_{M_n(\mathbb{C})}(a) \right) =\left\{ a \right\}.$$

We have recently generalized Nagy's result to the setting of atomic von Neumann algebras. We recall that a von Neumann algebra $M$ is called \emph{atomic} if it coincides with the weak$^*$ closure of the linear span of its minimal projections. It is known that every atomic von Neumann algebra $M$ can be written in the form $\displaystyle M = \bigoplus_j^{\ell_{\infty}} B(H_{j}),$ where each $H_j$ is a complex Hilbert space (compare \cite[\S V.1]{Tak} or \cite[\S 2.2]{S}).\smallskip

\begin{theorem}\label{t characterization of projection in terms of the sphere around a positive element}\cite[Theorem 2.3]{Per2017} Let $M$ be an atomic von Neumann algebra, and let $a$ be a positive norm-one element in $M$. Then the following statements are equivalent: \begin{enumerate}
[$(a)$] \item $a$ is a projection; \item $Sph^+_{M} \left( Sph^+_{M}(a) \right) =\{a\}$.
\end{enumerate}
\end{theorem}

Actually, if $a$ is a positive norm-one element in an arbitrary C$^*$-algebra $A$ satisfying $Sph^+_{A} \left( Sph^+_{A}(a) \right) =\{a\}$, then $a$ is a projection (see \cite[Proposition 2.2]{Per2017}).\smallskip

\begin{open problem}\label{open problem charact of projection} Does the equivalence in Theorem \ref{t characterization of projection in terms of the sphere around a positive element} hold when $M$ is a general von Neumann algebra or a C$^*$-algebra?
\end{open problem}

For a separable infinite dimensional complex Hilbert space $H_3$ and the C$^*$-algebra $K(H_3)$, of compact operators on $H_3$, we have actually established a more general result, whose finite dimensional version was given by G. Nagy in \cite[Proof of Claim 1]{Nagy2017}.

\begin{theorem}\label{t bi spherical set K(H)}\cite[Theorem 3.3]{Per2017} Let $H_3$ be a separable infinite dimensional complex Hilbert space. Then the identity $$Sph^+_{K(H_3)} \left( Sph^+_{K(H_3)}(a) \right) =\left\{ b\in S(K(H_3)^+) : \!\! \begin{array}{c}
    s_{_{K(H_3)}} (a) \leq s_{_{K(H_3)}} (b),  \hbox{ and }\\
     \textbf{1}-r_{_{B(H_3)}}(a)\leq \textbf{1}-r_{_{B(H_3)}}(b)
  \end{array}\!\!
 \right\},$$ holds for every $a$ in the unit sphere of $K(H_3)^+$.
\end{theorem}

A consequence of the above theorem gives an appropriate version of Theorem \ref{t characterization of projection in terms of the sphere around a positive element} for $K(H_3)$.

\begin{theorem}\label{t characterization of projection in terms of the sphere around a positive element k(H)}\cite[Theorem 2.5]{Per2017} Let $a$ be a positive norm-one element in $K(H_3)$, where $H_3$ is a separable complex Hilbert space. Then the following statements are equivalent: \begin{enumerate}
[$(a)$] \item $a$ is a projection; \item $Sph^+_{K(H_3)} \left( Sph^+_{K(H_3)}(a) \right) =\{a\}$.
\end{enumerate}
\end{theorem}

Thanks to Theorems \ref{t characterization of projection in terms of the sphere around a positive element} and \ref{t characterization of projection in terms of the sphere around a positive element k(H)} it can be concluded that given two atomic von Neumann algebras $M$ and $N$ (respectively, separable complex Hilbert spaces $H_3$ and $H_4$), and a surjective isometry $\Delta : S(M^+)\to S(N^+)$ (respectively, $\Delta : S(K(H_3)^+)\to S(K(H_4)^+)$), then $\Delta$ maps $\mathcal{P}roj(M)\backslash\{0\}$ onto $\mathcal{P}roj(N)\backslash\{0\}$ (respectively, $\mathcal{P}roj(K(H_3))\backslash\{0\}$ onto $\mathcal{P}roj(K(H_4))\backslash\{0\}$), and the restriction $$\Delta|_{\mathcal{P}roj(M)\backslash\{0\}} : \mathcal{P}roj(M)\backslash\{0\}\to \mathcal{P}roj(N)\backslash\{0\}$$ (respectively, $\Delta|_{\mathcal{P}roj(K(H_3))\backslash\{0\}} : \mathcal{P}roj(K(H_3))\backslash\{0\}\to \mathcal{P}roj(K(H_4))\backslash\{0\}$) is a surjective isometry.

These are some of the tools that combined with many other technical arguments are applied to give a partial solution to Problem \ref{problem positive Tingley's problem} in the setting of compact operators.

\begin{theorem}\label{t Nagy for K(ell2)}\cite[Theorem 3.7]{Per2017} Let $H_3$ and $H_4$ be separable complex Hilbert spaces. Let us assume that $H_3$ is infinite dimensional. We suppose that $\Delta : S(K(H_3)^+)\to S(K(H_4)^+)$ is a surjective isometry. Then there exists a surjective complex linear isometry $T: K(H_3)\to K(H_4)$ satisfying $T(x) = \Delta(x),$ for all $x\in S(K(H_3)^+)$. We can further conclude that $T$ is a $^*$-isomorphism or a $^*$-anti-isomorphism.
\end{theorem}

Additional technical results are given in \cite[\S 4]{Per2017} to give a complete solution to Problem \ref{problem positive Tingley's problem} in the setting of atomic von Neumann algebras. For brevity we shall not comment some of the deep technical results required to establish this solution. The final result reads as follows:

\begin{theorem}\label{t positive Tigley for B(H)}\cite[Theorem 4.5]{Per2017} Let $\Delta : S(B(H_1)^+)\to S(B(H_2)^+)$ be a surjective isometry, where $H_1$ and $H_2$ are complex Hilbert spaces. Then there exists a surjective complex linear isometry {\rm(}actually, a $^*$-isomorphism or a $^*$-anti-automorphism{\rm)} $T: B(H_1)\to B(H_2)$ satisfying $\Delta (x) = T(x),$ for all $x\in S(B(H_1)^+)$.
\end{theorem}

\smallskip\smallskip

\begin{open problem}\label{open problem Tingley positive} Let $\Delta : S(A^+)\to S(B^+)$ be a surjective isometry, where $A$ and $B$ are C$^*$-algebras. Does $\Delta$ admits an extension to a surjective complex linear isometry from $A$ onto $B$?
\end{open problem}

\begin{open problem}\label{open problem Tingley Cp} Let $H$ be an arbitrary complex Hilbert space, and let $p\in (1,\infty)$.
Suppose $\Delta : S(C_p(H)^+) \to S(C_p(H)^+)$ is a surjective isometry. Does $\Delta$ admit a unique extension to a surjective real linear isometry on $C_p(H)$.
\end{open problem}

A more general version has been also posed by M. Mori in \cite[Problem 6.3]{Mori2017}.

\begin{open problem}\label{open problem Tingley noncommutative Lp} Let $1 < p < \infty$, $p\neq 2$, let $M$, $N$ be von Neumann algebras and $\Delta : S(L^p(M))\to S(L^p(N))$ be a surjective isometry between the unit spheres of two noncommutative $L^p$-spaces (with respect to fixed normal semifinite faithful weights). Does $\Delta$ admit an extension to a real linear surjective isometry $T : L^p(M)\to L^p(N)$?
\end{open problem}

\medskip

\textbf{Acknowledgements} Authors partially supported by the Spanish Ministry of Economy and Competitiveness (MINECO) and European Regional Development Fund project no. MTM2014-58984-P and Junta de Andaluc\'{\i}a grant FQM375.\smallskip

I thank the organizers of the meeting \emph{``Preservers Everywhere, Szeged-2017''} for a successful and fruitful initiative.

\end{document}